\documentclass{article}

\usepackage[final]{neurips_2020}

\usepackage[utf8]{inputenc} 
\usepackage[T1]{fontenc}    
\usepackage{hyperref}       
\usepackage{url}            
\usepackage{booktabs}       
\usepackage{amsfonts}       
\usepackage{nicefrac}       
\usepackage{microtype}      
\usepackage{caption}
\usepackage{subcaption}

\usepackage{natbib}
\usepackage{amsmath,amssymb,amsthm}
\usepackage{enumitem}
\usepackage[x11names]{xcolor}
\usepackage{graphicx}

\newcommand{\IID}{\stackrel{IID}{\sim}} 

\newcommand{\F}{\mathcal{F}} 
\newcommand{\X}{\mathcal{X}} 
\renewcommand{\L}{\mathcal{L}} 
\newcommand{\C}{\mathcal{C}} 
\newcommand{\B}{\mathcal{B}} 
\newcommand{\R}{\mathbb{R}} 
\newcommand{\N}{\mathbb{N}} 
\newcommand{\Z}{\mathbb{Z}} 

\newcommand{\E}{\mathop{\mathbb{E}}} 
\renewcommand{\P}{\mathcal{P}} 
\newcommand{\norm}[1]{\left\lVert#1\right\rVert} 
\newcommand{\M}{\mathcal{M}} 
\renewcommand{\hat}{\widehat}
\renewcommand{\tilde}{\widetilde}

\newtheorem{theorem}{Theorem}

\newtheorem{corollary}[theorem]{Corollary}
\newtheorem{lemma}[theorem]{Lemma}
\newtheorem{definition}[theorem]{Definition}

\title{Robust Density Estimation under Besov IPMs}

\author{%
  Ananya Uppal\\
  Department of Mathematical Sciences\\
  Carnegie Mellon University\\
  \texttt{auppal@andrew.cmu.edu} \\
   \And
   Shashank Singh\\
   Machine Learning Department\\
   Carnegie Mellon University\\
   \texttt{shashanksi@google.com} \\
   \And
   Barnab\'as P\'oczos\\
   Machine Learning Department\\
   Carnegie Mellon University\\
   \texttt{bapoczos@cs.cmu.edu} \\
}

\begin{document}

\maketitle

\begin{abstract}
    We study minimax convergence rates of nonparametric density estimation under the Huber contamination model, in which a proportion of the data comes from an unknown outlier distribution. We provide the first results for this problem under a large family of losses, called Besov integral probability metrics (IPMs), that include the $\L^p$, Wasserstein, Kolmogorov-Smirnov, Cramer-von Mises, and other commonly used metrics. Under a range of smoothness assumptions on the population and outlier distributions, we show that a re-scaled thresholding wavelet estimator converges at minimax optimal rates under a wide variety of losses and also exhibits optimal dependence on the contamination proportion. We also provide a purely data-dependent extension of the estimator that adapts to both an unknown contamination proportion and the unknown smoothness of the true density. Finally, based on connections recently shown between density estimation under IPM losses and generative adversarial networks (GANs), we show that certain GAN architectures are robustly minimax optimal.
\end{abstract}

\section{Introduction}
    
    In many settings, observed data contains not only samples from the distribution of interest, but also a small proportion of outlier samples. Because these outliers can exhibit arbitrary, unpredictable behavior, they can be difficult to detect or to explicitly account for. This has inspired a large body of work on \emph{robust statistics}, which seeks statistical methods for which the error introduced by a small proportion of arbitrary outlier samples can be controlled.
    
    The majority of work in robust statistics has focused on providing guarantees under the Huber $\epsilon$-contamination model~\citep{huber1965robust}. Under this model, data is assumed to be observed from a mixture distribution $(1 - \epsilon) P + \epsilon G$, where $P$ is an unknown population distribution of interest, $G$ is an unknown outlier distribution, and $\epsilon \in [0,1)$ is the ``contamination proportion'' of outlier samples. Equivalently, this models the misspecified case in which data are drawn from a small perturbation by $\epsilon(G - P)$ of the target distribution $P$ of interest. The goal is then to develop methods whose performance degrades as little as possible when $\epsilon$ is non-negligible.
    
    The present paper studies nonparametric density estimation under this model. Specifically, given independent and identically distributed samples from the mixture $(1 - \epsilon) P + \epsilon G$, we characterize minimax optimal convergence rates for estimating $P$. Prior work on this problem has assumed $P$ has a H\"older continuous density $p$ and has provided minimax rates under total variation loss~\citep{chen2018robust} or for estimating $p(x)$ at a point $x$~\citep{liu2017density}. In the present paper, in addition to considering a much wider range of smoothness conditions (characterized by $p$ lying in a Besov space), we provide results under a large family of losses called integral probability metrics (IPMs);
    \begin{equation}
        d_\F(P, Q) = \sup_{f\in \F}
            \left|
                \E_{X\sim P}f(X)-\E_{X\sim Q}f(X)
            \right|,
        \label{eq:IPM}
    \end{equation}
    where $P$ and $Q$ are probability distributions and $\F$ is a ``discriminator class'' of bounded Borel functions. As shown in several recent papers~\citep{liu2017approximationInGANs,liang2018well,singh2018adversarial,uppal2019nonparametric}, IPMs play a central role not only in nonparametric statistical theory and empirical process theory, but also in the theory of generative adversarial networks (GANs). Hence, this work advances not only basic statistical theory but also our understanding of the robustness properties of GANs.
    
    In this paper, we specifically discuss the case of Besov IPMs, in which $\F$ is a Besov space (see Section~\ref{sec:setup}).
    In classical statistical problems, Besov IPMs provide a unified formulation of a wide variety of distances, including $\L^p$ \citep{wasserman2006nonparametric,tsybakov2009introduction}, Sobolev \citep{mroueh2017sobolevGANs,leoni2017first},
    maximum mean discrepancy (MMD; \citep{tolstikhin2017minimax})/energy \citep{szekely2007distances,ramdas2017wasserstein}, 
    Wasserstein/Kantorovich-Rubinstein \citep{kantorovich1958space,villani2008optimal}, Kolmogorov-Smirnov \citep{kolmogorov1933sulla,smirnov1948table}, and Dudley metrics \citep{dudley1972speeds,abbasnejad2018deep}. Hence, as we detail in Section~\ref{sec:examples}, our bounds for robust nonparametric density estimation apply under many of these losses. More recently, it has been shown that generative adversarial networks (GANs) can be cast in terms of IPMs, such that convergence rates for density estimation under IPM losses imply convergence rates for certain GAN architectures~\citep{liang2018well,singh2018adversarial,uppal2019nonparametric}. Thus, as we show in Section~\ref{sec:GANs}, our results imply the first robustness results for GANs in the Huber model.
    
    In addition to showing rates in the classical Huber model, which avoids assumptions on the outlier distribution $G$, we consider how rates change under additional assumptions on $G$. Specifically, we show faster convergence rates are possible under the assumption that $G$ has a bounded density $g$, but that these rates are not further improved by additional smoothness assumptions on $g$.
    
    Finally, we overcome a technical limitation of recent work studying density estimation under Besov IPMs losses. Namely, the estimators used in past work rely on the unrealistic assumption that the practitioner knows the Besov space in which the true density lies. This paper provides the first convergence rates for a purely data-dependent density estimator under Besov IPMs, as well as the first nonparametric convergence guarantees for a fully data-dependent GAN architecture.

    \subsection{Paper Organization}
        The rest of this paper is organized as follows. Section~\ref{sec:notation} formally states the problem we study and defines essential notation. Section~\ref{sec:related_work} discusses related work in nonparametric density estimation. Section~\ref{sec:unstructured} contains minimax rates under the classical ``unstructured'' Huber contamination model, while Section~\ref{sec:structured} studies how these rates change when additional assumptions are made on the contamination distribution. Section~\ref{sec:examples} develops our general results from Sections~\ref{sec:unstructured} and \ref{sec:structured} into concrete minimax convergence rates for important special cases.
        Finally, Section~\ref{sec:GANs} applies our theoretical results to bound the error of perfectly optimized GANs in the presence of contaminated data. All theoretical results are proven in the Appendix.
\section{Formal Problem Statement}
    \label{sec:notation}

    We now formally state the problems studied in this paper.
    Let $p$ be a density of interest and $g$ be the contamination density such that $X_1, \dots, X_n\sim (1-\epsilon)p+\epsilon g$ are $n$ IID samples. We wish to use these samples to estimate $p$. We consider two qualitatively different types of contamination, as follows.
    
    In the ``unstructured'' or Huber contamination setting, we assume that $p$ lies in some regularity class $\F_g$, but $g$ may be any compactly supported density. In particular, we assume that the data is generated from a density living in the set $\M(\epsilon, \F_g)=\{(1-\epsilon)p+\epsilon g: p \in \F_g, g \text{ has compact support}\}$.
    We then wish to bound the minimax risk of estimating $p$ under an IPM loss $d_{\F_d}$; i.e., the quantity
    \begin{equation}
        \mathcal{R}(n, \epsilon,\F_g, \F_d) = \inf_{\hat{p}_n}\sup_{f\in \M(\epsilon, \F_g)} \E_{f} \left[ d_{\F_d}(p, \hat{p}_n) \right]
        \label{eq:unstructured_risk}
    \end{equation}
    where the infimum is taken over all estimators $\hat{p}_n$.
    
    In the ``structured'' contamination setting, we additionally assume that the contamination density $g$ lives in a smoothness class $\F_c$. The data is generated by a density in $\M(\epsilon, \F_g, \F_c) = \{(1-\epsilon)p +\epsilon g: p\in \F_g, g\in \F_c\}$ and we seek to bound the minimax risk
    \begin{equation}
        \mathcal{R}(n, \epsilon, \F_g, \F_c, \F_d) = \inf_{\hat{p}_n}\sup_{f\in \M(\epsilon, \F_g, \F_c)} \E_f \left[ d_{\F_d}(\hat{p}_n, p) \right].
        \label{eq:structured_risk}
    \end{equation}
    In the following section, we provide notation to formalize the spaces $\F_g, \F_c$ and $\F_d$ that we consider.
    
    \subsection{Set up and Notation}
    \label{sec:setup}
    For non-negative real sequences $\{a_n\}_{n \in \N}$, $\{b_n\}_{n \in \N}$, $a_n \lesssim b_n$ indicates $\limsup_{n \to \infty} \frac{a_n}{b_n} < \infty$, and $a_n~\asymp~b_n$ indicates $a_n \lesssim b_n \lesssim a_n$.
    For $q \in [1,\infty]$, $q' := \frac{q}{q - 1}$ denotes the H\"older conjugate of $q$ (with $1' = \infty$, $\infty' = 1$). $\L^q(\R^D)$ (resp. $l^q$) denotes the set of functions $f$ (resp. sequences $a$) with $\norm{f}_q := \left(\int |f(x)|^q \, dx \right)^{1/q}<\infty$ (resp. $\norm{a}_{l^q} := \left(\sum_{n\in \N}|a_n|^q\right)^{1/q}<\infty$).
    
    We now define the family of Besov spaces studied in this paper. Besov spaces generalize H\"older and Sobolev spaces and are defined using wavelet bases. As opposed to the Fourier basis, wavelet bases provide a localization in space as well as frequency which helps express spatially inhomogeneous smoothness. 
    
    A wavelet basis is formally defined by the mother ($\psi(x)$) and father($\phi(x)$) wavelets. The basis consists of two parts; first, the set of translations of the father and mother wavelets i.e. 
    \begin{align}
        \Phi &=\{\phi(x-k):k\in \Z^d\}\\
        \Psi &=\{\psi_\epsilon(x-k):k\in \Z^d, \epsilon \in \{0,1\}^D\},
    \end{align}
    second, the set of daughter wavelets, i.e.,
    \begin{align}
        \Psi_j &=\{ 2^{Dj/2} \psi_\epsilon(2^{Dj}x-k):k\in \Z^d, \epsilon \in \{0,1\}^D\}.
    \end{align}
    Then the union $\Phi\cup\Psi\cup(\bigcup_{j\geq0})\Psi_j$ is an orthonormal basis for $\L^2(\R^D)$.

    We defer the technical assumptions on the mother and father wavelet to the appendix. Instead for intuition, we illustrate in Figure~\ref{fig:haar} the few terms of the best-known wavelet basis of $\L^2(\R)$, the Haar wavelet basis.
    \begin{figure}
        \centering
        \includegraphics[width=0.95\linewidth]{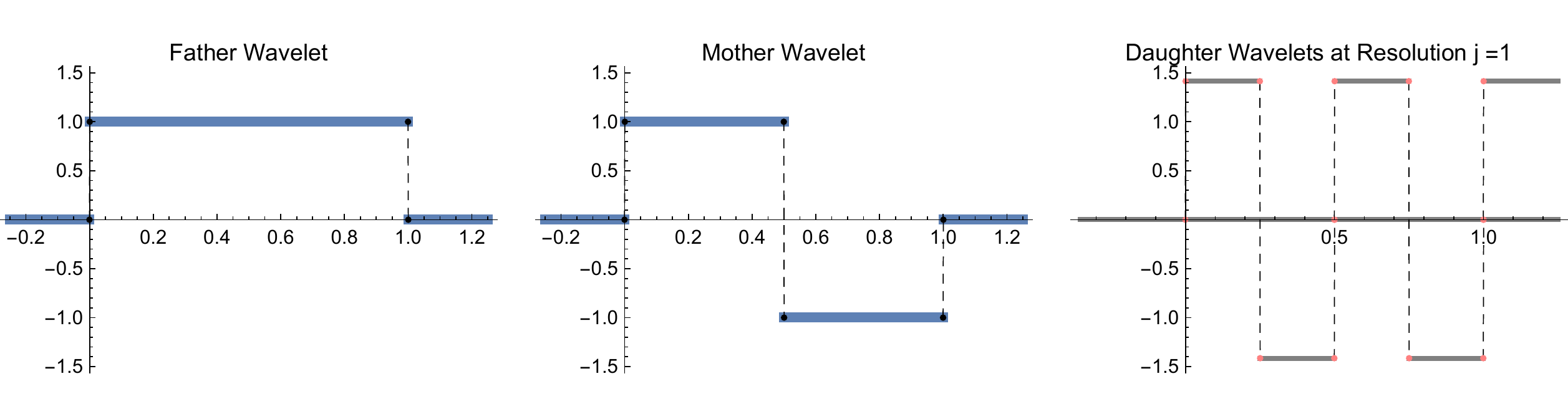}
        \caption{Father, Mother, and first few Daughter elements of the Haar Wavelet Basis.}
        \label{fig:haar}
            \vspace{-3ex}
    \end{figure}
    
    In higher dimensions, the wavelet basis is defined using the tensor product of wavelets in dimension 1. For details, see, \cite{hardle2012wavelets} and \cite{meyer1992wavelets}.
     
     To effectively express smooth functions we will require $r$-regular ($r$-regularity is precisely defined in the appendix) wavelets. We assume throughout our work that $\phi$ and $\psi$ are compactly supported $r$-regular wavelets.  We now formally define a Besov space.

    \begin{definition}[Besov Space] Given an $r$-regular wavelet basis of $\L^2(\R^D)$, let $0 \leq \sigma < r$, and $p,q \in [1,\infty]$. Then the \emph{Besov space} $B^{\sigma}_{p,q}(\R^D)$ is defined as the set of functions $f : \R^D \to \R$ that satisfy
	\begin{align}
	     \norm{f}_{B_{p,q}^\sigma} 
	     := \norm{\alpha}_{l^p}+
	    \norm{
	        \left\{
	            2^{j(\sigma+D(1/2-1/p))}
	            \norm{\beta_j}_{l^p}
	      \right\}_{j \in \N}}_{l^q}
	    <\infty
	\end{align}
	where $\alpha$ is the set of vectors $\{\alpha_{\phi}\}_{\phi\in \Phi}$ where $\alpha_\phi := \int_{\R^D} f(x)\phi(x)dx$ and $\beta_j$ is the set of vectors $\{\beta_{\psi}\}_{\psi \in \Psi_j}$, where $\beta_{\psi} := \int_{\R^D} f(x) \psi(x) dx$.
    \end{definition}
    The quantity $\|f\|_{B^\sigma_{p,q}}$ is called the \emph{Besov norm of $f$}. For any $L > 0$, we write $B^\sigma_{p,q}(L)$ to denote the closed Besov ball $B^\sigma_{p,q}(L) = \{f \in B^\sigma_{p,q} : \|f\|_{B^\sigma_{p,q}} \leq L\}$.
	When the constant $L$ is unimportant (e.g., for \emph{rates} of convergence), $B^\sigma_{p,q}$ denotes a ball $B_{p,q}^{\sigma}(L)$ of finite but arbitrary radius $L$.
    We provide well-known examples from the rich class of resulting spaces in Section~\ref{sec:examples}.

    We now define ``linear (distribution) estimators'', a commonly used sub-class of distribution estimators:
    \begin{definition}[Linear Estimator]
    \label{def:linear_est}
    Let $(\Omega, \F, P)$ be a probability space. An estimate $\hat P$ of $P$ is said to be \emph{linear} if there exist functions $T_i(X_i,\cdot) : \F \to \R$ such that for all measurable $A \in \F$,
    $
        \hat{P}(A) = \sum_{i = 1}^n T_i(X_i,A).
    $
    \end{definition}
    
    Common examples of linear estimators are the empirical distribution, the kernel density estimator and the linear wavelet series estimator considered in this paper.
\section{Related Work}
\label{sec:related_work}

    This paper extends recent results in both non-parametric density estimation and robust estimation.
    We now summarize the results of the most relevant papers, namely those of \citet{uppal2019nonparametric}, \citet{chen2016general}, and \citet{liu2017density}.
    
    \subsection{Nonparametric Density Estimation under Besov IPM Losses}

    \citet{uppal2019nonparametric} studied the estimation of a density lying in a Besov space $B_{p_g,q_g}^{\sigma_g}$ under Besov IPM loss $d_{B_{p_d,q_d}^{\sigma_d}}$ with uncontaminated data. As shorthand, we will write $\mathcal{R}(n, \F_g, \F_d) = \mathcal{R}(n, 0,\F_g, \F_d)$ and $\mathcal{R}(n, \F_g, \F_c, \F_d) = \mathcal{R}(n, 0, \F_g, \F_c, \F_d)$ to denote the corresponding uncontaminated rates derived by \citet{uppal2019nonparametric}. They used the wavelet thresholding estimator, proposed in \cite{donoho1996density}, to derive a minimax convergence rate of the form
    \begin{equation}
        \mathcal{R} \left( n, \F_g, \F_d \right)
        = n^{-1/2}
        + n^{-\frac{\sigma_g+\sigma_d}{2\sigma_g+D}}
        + n^{-\frac{\sigma_g+\sigma_d - D/p_g + D/p_d'}{2\sigma_g + D \left( 1 - 2/p_g \right)}},
        \label{eq:uncontaminated_minimax_rate}
    \end{equation}
    (omitting polylog factors in $n$).
    Extending a classical result of \citet{donoho1996density}, they also showed that, if the estimator is restricted to be linear (in the sense of Def.~\ref{def:linear_est}), then the minimax rate slows to
    \begin{equation}
        \mathcal{R}_L \left( n, \F_g, \F_d \right)
        = n^{-1/2}
        + n^{-\frac{\sigma_g+\sigma_d}{2\sigma_g+D}}
		+ n^{-\frac{\sigma_g+\sigma_d-D/p_g+D/p_d'}{2\sigma_g+D(1-2/p_g+2/p_d')}}.
        \label{eq:linear_uncontaminated_minimax_rate}
    \end{equation}
    The first two terms in ~\eqref{eq:uncontaminated_minimax_rate} \& \eqref{eq:linear_uncontaminated_minimax_rate} are identical and the third term in \eqref{eq:linear_uncontaminated_minimax_rate} is slower. In particular, when $p_d' > p_g$ and $\sigma_d < D/2$, linear estimators are strictly sub-optimal, while the wavelet thresholding estimator converges at the optimal rate. The present paper extends this work in two directions.
    
    First, we study how the minimax risk of estimating the data density $p$ changes when the observed data are contaminated by a proportion $\epsilon$ of outliers from a (potentially adversarially chosen) contamination distribution $g$. We show that, in most cases, wavelet thresholding estimators remain minimax optimal under both structured and unstructured contamination settings. Moreover, for $p_d'\leq p_g$ linear wavelet estimators are minimax optimal under the structured contamination setting and the unstructured contamination setting if the IPM is generated by a smooth enough class of functions ($\sigma_d\geq D/p_d$).
    
    Second, noting that the estimators of \citet{uppal2019nonparametric} rely on knowledge of the smoothness parameter $\sigma_g$ of the true density, we consider the more realistic case where $\sigma_g$ is unknown. We develop a fully data-dependent variant of the wavelet thresholding estimator from \citet{uppal2019nonparametric} that is minimax optimal for all $\sigma_g$ under structured contamination.
        
    Finally, \citet{uppal2019nonparametric} also applied their results to bound the risk of a particular generative adversarial network (GAN) architecture.
    They show that the GAN is able to learn Besov densities at the minimax optimal rate. In this paper, we show that the same GAN architecture continues to be minimax optimal in the presence of outliers, and that, with minor modifications, it can do so without knowledge of the smoothness $\sigma_g$ of the true density.
    
    \subsection{Nonparametric Density Estimation with Huber Contamination}
    
    \citet{chen2016general} give a unified study of a large class of robust nonparametric estimation problems under the total variation loss. In the particular case of estimating a $\sigma_g$-H\"older continuous density, their results imply a minimax convergence rate of $n^{-\frac{\sigma_g}{2\sigma_g+1}}+\epsilon$,
    matching our results (theorem~\ref{thm:non_linear_unstructured_rate}) for total variation loss.
    The results of \citet{chen2016general} are quite specific to total variation loss, whereas, we provide results for a range of loss functions as well as densities of varying smoothness.
    Moreover, the estimator studied by \citet{chen2016general} is not computable in practice. It involves solving a testing problem between all pairs of points in a total variation cover of the hypothesis class in which the true density is assumed to lie. In contrast, our upper bounds rely on a simple thresholded wavelet series estimator, which can be computed in linear time (in the sample size $n$) with a fast wavelet transform.
    
    \citet{liu2017density} studied $1$-dimensional density estimation at a point $x$ (i.e., estimating $p(x)$ instead of the entire density $p$) for H\"older smoothness densities under the Huber $\epsilon$-contamination model. In the case of unstructured contamination (arbitrary $G$), \citet{liu2017density} derived a minimax rate of
    \begin{equation}
        n^{-\frac{\sigma_0}{2\sigma_0 + 1}} + \epsilon^{\frac{\sigma_0}{\sigma_0 + 1}}
        \label{rate:liu_gao}
    \end{equation}
    in root-mean-squared error. With the caveats that we study estimation of the entire density $p$ rather than a single point $p(x)$ and assume that $G$ has a density $g$, this corresponds to our setting when $p_g=q_g=\infty$, and $D=1$. Our results (equation~\ref{rate:unstructured_dense_rate}) imply an upper bound on the rate of
    \begin{equation}
        n^{-\frac{\sigma_0}{2\sigma_0+1}} + \epsilon^{\frac{\sigma_0}{\sigma_0+(1 - 1/p)}}
        \label{rate:liu_gao_ours}
    \end{equation}
    under the $\L^p$ loss. Interestingly, this suggests that estimating a density at a point under RMSE is harder than estimating an entire density under $\L^2$ loss, and is, in fact, as hard as estimation under $\L^\infty$ ($\sup$-norm) loss. While initially perhaps surprising, this makes sense if one thinks of rates under $\L^\infty$ loss as being the rate of estimating the density at the worst-case point over the sample space, which may be the point $x$ at which \citet{liu2017density} estimate $p(x)$; under minimax analysis, these become similar.
    
    We generalize these rates to (a) dimension $D > 1$, (b) densities $p$ lying in Besov spaces $B_{p_g,q_g}^{\sigma_g}$, and (c) a wide variety of losses parametrized by Besov IPMs ($B_{p_d,q_d}^{\sigma_d}$).
    
    \citet{liu2017density} also study the case of structured contamination, in which $g$ is assumed to be $\sigma_c$-H\"older continuous. Because they study estimation at a point, their results depend on an additional parameter, denoted $m$, which bounds the value of the contamination density $g$ at the target point (i.e., $g(x) \leq m$). They derive a minimax rate of
    \begin{equation}
        n^{-\frac{\sigma_g}{2\sigma_g + 1}} + \epsilon \min\{1, m\} + n^{-\frac{\sigma_c}{2\sigma_c + 1}} \epsilon^{-\frac{\sigma_c}{2\sigma_c + 1}}.
    \end{equation}
    This rate contains a term depending only on $n$ that is identical to the minimax rate in the uncontaminated case, a term depending only on $\epsilon$, and a third ``mixed'' term.
    Notably, one can show that this mixed term $n^{-\frac{\sigma_c}{2\sigma_c + 1}} \epsilon^{-\frac{\sigma_c}{2\sigma_c + 1}}$ is always dominated by $n^{-\frac{\sigma_g}{2\sigma_g + D}} + \epsilon$, and so, unless $m \to 0$ as $n \to \infty$, the mixed term is negligible. In this paper, because we study estimation of the entire density $p$, the role of the parameter $m$ is played by $M := \|g\|_\infty$. Since $g$ is assumed to be a density with bounded support, we cannot have $M \to 0$; thus, in our results, the mixed term does not appear. Aside from this distinction, our results (Theorem~\ref{thm:structured_rate}) again generalize the results of \citet{liu2017density} to higher dimensions, other Besov classes of densities, and new IPM losses.
        
    Finally, we mention two early papers on robust nonparametric density estimation by \citet{kim2012robust} and \citet{vandermeulen2013consistency}. These papers introduced variants of kernel density estimation based on $M$-estimation, for which they demonstrated robustness to arbitrary contamination using influence functions. These estimators are more complex than the scaled series estimates we consider, in that they non-uniformly weight the kernels centered at different sample points. While they also showed $\L^1$ consistency of these estimators, they did not provide rates of convergence, and so it is not clear when these estimators are minimax optimal.
\section{Minimax Rates}
        Here we give our main minimax bounds. First, we state the estimators used for the upper bounds.
        \paragraph{Estimators:} To illustrate the upper bounds we consider two estimators that have been widely studied in the uncontaminated setting (see \cite{donoho1996density}, \cite{uppal2019nonparametric}) namely the wavelet thresholding estimator and the linear wavelet estimator. All bounds provided here are tight up to polylog factors of $n$ and $1/\epsilon$.

    For any $j_1\geq j_0\geq 0$ the wavelet thresholding estimator is defined as
        \begin{equation}
    		\hat{p}_n
    		=   \sum_{ \phi\in \Phi} \hat{\alpha}_{\phi} \phi+
    		    \sum_{j=0}^{j_0} \sum_{\psi\in \Psi_j} \hat{\beta}_{\psi}\psi +
    		    \sum_{j=j_0}^{j_1} \sum_{\psi\in \Psi_j} \tilde{\beta}_{\psi}\psi
    		    \label{eq:nonlinear_wavelet_estimator}
    	\end{equation}
	where $\hat{\alpha}_{\phi} = \frac{1}{n}\sum_{i=1}^n \phi(X_i)$ and $\hat{\beta}_{\psi} = \frac{1}{n}\sum_{i=1}^n \psi(X_i)$ and coefficients of some of the wavelets with higher resolution (i.e., $j \in [j_0, j_1]$) are hard-thresholded: $\tilde{\beta}_{\psi} = \hat{\beta}_\psi 1_{\hat{\beta}_\psi\geq t}$ for threshold $t = c\sqrt{j/n}$, where $c$ is a constant. 
	
	The linear wavelet estimator is simply $\hat{p}_n$ with only linear terms (i.e., $j_0=j_1$).
	Here $j_0, j_1$ correspond to smoothing parameters which we carefully choose to provide upper bounds on the risk. In the sequel, let $\F_g=B^{\sigma_g}_{p_g, q_g}(L_g)$ and $\F_d = B^{\sigma_d}_{p_d, q_d}(L_d)$ be Besov spaces. 
\subsection{Unstructured Contamination}
\label{sec:unstructured}
    In this section we consider the density estimation problem under Huber's $\epsilon$ contamination model; i.e. we have no structural assumptions on the contamination. Let 
        $
            X_1, \dots, X_n \IID (1-\epsilon)p+\epsilon g
        $,
    where $p$ is the true density and $g$ is any compactly supported probability density. We provide bounds on the minimax risk of estimating the density $p$. We let
    \begin{equation}
         \M(\epsilon, \F_g)  =\{(1-\epsilon)p+\epsilon g: p\in \F_g, g \text{ has compact support}\}
    \end{equation}
    and bound the minimax risk 
    \begin{equation}
        \mathcal{R}(n, \epsilon, \F_g, \F_d)=
        \inf_{\hat{p}}\sup_{f\in \M(\epsilon, \F_g)} \E_f d_{\F_d}(\hat{p},p)
    \end{equation}
    where the infimum is taken over all estimators $\hat{p}_n$ constructed from the $n$ IID samples.
	
	We first present our results for what \citet{uppal2019nonparametric} called the ``Sparse'' regime $p_d'\geq p_g$, in which the worst-case error is caused by large ``spikes'' in small regions of the sample space. Within this ``Sparse'' regime, we are able to derive minimax convergence rates for all Besov spaces $\F_g=B^{\sigma_g}_{p_g, q_g}(L_g)$ and $\F_d = B^{\sigma_d}_{p_d, q_d}(L_d)$. Surprisingly, we find that linear and nonlinear estimators have identical, rate-optimal dependence on the contamination proportion $\epsilon$ in this setting. Consequently, if $\epsilon$ is sufficiently large,
	then the difference in asymptotic rate between linear and nonlinear estimators vanishes.
    We first show the minimax rate in this setting that is achieved by a scaled version wavelet thresholding estimator i.e. $\frac{1}{(1-\epsilon)}\hat{p}_n$. The proof is provided in section B.1 of the appendix. 
    \begin{theorem}(\textbf{Minimax Rate, Sparse Case})
        \label{thm:non_linear_unstructured_rate}
        Let $r>\sigma_g> D/p_g$ and $p_d'\geq p_g$. Then,
        \begin{align}
		    \mathcal{R} \left(n, \epsilon, B_{p_g,q_g}^{\sigma_g}, B_{p_d,q_d}^{\sigma_d} \right)
		    \sim \mathcal{R} \left( n, B_{p_g,q_g}^{\sigma_g}, B_{p_d,q_d}^{\sigma_d} \right)
		    + \epsilon
			+\epsilon^{\frac{\sigma_g+\sigma_d+D/p_d'-D/p_g}{\sigma_g-D/p_g+D}}
			\label{rate:non_linear}
	    \end{align}
    \end{theorem}
    On the other hand linear estimators are only able to achieve the following asymptotic rate. The proof is provided in section B.2 of the appendix. 
    \begin{theorem}(\textbf{Linear Minimax Rate, Sparse Case})
        \label{thm:linear_unstructured_rate}
        Let $r>\sigma_g> D/p_g$ and $p_d'\geq p_g$. Then,
		\begin{align}
			\mathcal{R}_L \left(n, \epsilon, B_{p_g,q_g}^{\sigma_g}, B_{p_d,q_d}^{\sigma_d} \right)
			\sim \mathcal{R}_L \left(n, B_{p_g,q_g}^{\sigma_g}, B_{p_d,q_d}^{\sigma_d} \right)
			+ \epsilon
			+ \epsilon^{\frac{\sigma_g+\sigma_d+D/p_d'-D/p_g}{\sigma_g-D/p_g+D}}
		\end{align}
    \end{theorem}
    As is expected, the sub-optimality of linear estimators referred to in section~\ref{sec:related_work}, extends to the contaminated setting when contamination $\epsilon$ is small.
    However, if the contamination $\epsilon$ is large the distinction between linear and non-linear estimators disappears. More specifically, if  $\epsilon$ is large enough then both estimators converge at the same rate of $\epsilon+ \epsilon^{\frac{\sigma_g+\sigma_d+D/p_d'-D/p_g}{\sigma_g-D/p_g+D}}$.
    
    \paragraph{Bounds for the regime $\boldsymbol{p_d'\leq p_g}$:}
    We note that the lower bounds that constitute the minimax rates above hold for all values of $p_g, p_d'\geq 1$. Furthermore, the linear wavelet estimator implies an upper bound (shown in section B.3 of the appendix) on the risk in the dense regime. Together, this gives, for all $r>\sigma_g>D/p_g$ and $p_d'\leq p_g$, 
    \begin{align}
        \label{rate:unstructured_dense_rate}
        \Delta(n)
        + \epsilon 
        + \epsilon^{\frac{\sigma_g+\sigma_d+D/p_d'-D/p_g}{\sigma_g-D/p_g+D}}
        \leq 
        \mathcal{R} \left( n,\epsilon, B_{p_g,q_g}^{\sigma_g}, B_{p_d,q_d}^{\sigma_d} \right) 
        \leq 
        \Delta(n)
        + \epsilon 
        + \epsilon^{\frac{\sigma_g+\sigma_d}{\sigma_g+D/p_d}}
    \end{align}
    where $\Delta(n)=\mathcal{R} \left( n, B_{p_g,q_g}^{\sigma_g}, B_{p_d,q_d}^{\sigma_d} \right)$.
    
    One can check that, when the discriminator is sufficiently smooth (specifically, $\sigma_d\geq D/p_d$), the term $\epsilon^{\frac{\sigma_g+\sigma_d}{\sigma_g+D/p_d}}$ on the right-hand side of Eq.~\eqref{rate:unstructured_dense_rate} is dominated by $\epsilon$; hence, the lower and upper bounds in Eq.~\eqref{rate:unstructured_dense_rate} match and the thresholding wavelet estimator is minimax rate-optimal. When $\sigma_d < D/p_d$, a gap remains between our lower and upper bounds, and we do not know whether the thresholding wavelet estimator is optimal.
    The sample mean is generally well-known to be sensitive to outliers in the data, and a large amount of recent work \citep{lugosi2016risk,lerasle2018monk,minsker2019distributed,diakonikolas2019robust} has proposed estimators that might be better predictors of the mean in the case of contamination by outliers. Since the linear and thresholding wavelet estimators are both functions of the empirical means $\hat{\beta}_\psi$ of the wavelet basis functions, we conjecture that a density estimator based on a better estimate of the wavelet mean $\beta_\psi^p$ might be able to converge at a faster rate as $\epsilon \to 0$. We leave this investigation for future work.

\subsection{Structured Contamination}
    \label{sec:structured}
    In the previous section, we analyzed minimax rates without any assumptions on the outlier distribution. In certain settings, this may be an overly pessimistic contamination model, and the outlier distribution may in fact be somewhat well-behaved. In this section, we study the effects of assuming the contamination distribution $G$ has a density $g$ that is either bounded or smooth. Our results show that assuming boundedness of $g$ improves the dependence of the minimax rate on $\epsilon$ to order $\asymp \epsilon$, but assuming additional smoothness of $g$ does not further improve rates.
    
    As described in Section~\ref{sec:notation}, in this setting we consider a more general form of the minimax risk:
    \begin{equation}
        \mathcal{R} (n, \epsilon, \F_g, \F_c, \F_d )
        = \inf_{\hat{p}}\sup_{f\in \M(\epsilon, \F_g, \F_c)}\E_f [d_{\F_d}(\hat{p},p)]
        \label{def:structured_minimax_risk}
    \end{equation}
    The additional parameter $\F_c$ denotes the class of allowed contamination distributions.
    
    
    We provide the following asymptotic rate for the above minimax risk that is achieved by an adaptive wavelet thresholding estimator with $2^{j_0} = n^{\frac{1}{2r+D}}$ and $2^{j_1} = (n/\log n)^{1/D}$. Recall here that $r$ is the regularity of the wavelets. Thus, for any $\sigma_g<r$, this estimator does not require the knowledge of $\sigma_g$.
    
    \begin{theorem}[Minimax Rate under Structured Contamination]
        \label{thm:structured_rate}
        Let $\sigma_g\geq D/p_g$, $\sigma_c > D/p_c$ and $\epsilon\leq 1/2$. Then, up to poly logarithmic factors of $n$,
        \begin{align}
    			\mathcal{R} \left( n, \epsilon,B_{p_g,q_g}^{\sigma_g}, B_{p_c,q_c}^{\sigma_c}, B_{p_d,q_d}^{\sigma_d} \right)
    			\asymp \mathcal{R} \left( n, \epsilon, B_{p_g,q_g}^{\sigma_g}, \L^\infty, B_{p_d,q_d}^{\sigma_d} \right)
    			\asymp \mathcal{R} \left( n, B_{p_g,q_g}^{\sigma_g}, B_{p_d,q_d}^{\sigma_d} \right)
    			+ \epsilon
	    \end{align}
    \end{theorem}
    
    
    The right-most term is simply $\epsilon$ plus the rate in the absence of contamination.    
    The left two terms are the rates when the contamination density lies, respectively, in the Besov space $B_{p_c,q_c}^{\sigma_c}$ and the space $\L^{\infty}$ of essentially bounded densities.
    In particular, these rates are identical when $\sigma_c > D/p_c$. One can check (see Lemma 10 in the Appendix) that, if $\sigma_c > D/p_c$, then $B_{p_c,q_c}^{\sigma_c} \subseteq \L^\infty$. Hence, Theorem~\ref{thm:structured_rate} shows that assuming boundedness of the contamination density improves the dependence on $\epsilon$ (compared to unstructured rates from the previous section), but that additional smoothness assumptions do not help. 

    In section B.1 of the appendix we first provided a proof of the upper bound using the classical wavelet thresholding estimator and then show the optimality of the adaptive version in section B.4.   
\subsection{Examples}
\label{sec:examples}
    
    Here, we summarize the implications of our main results for robust density estimation in a few specific examples, allowing us to directly compare with previous results.

        The case $p_d = q_d = \infty$,
        includes, as examples,
        the total variation loss $d_{B_{p_d,q_d}^0}$ ($\sigma_d=0$, \cite{rudin2006real}) and
        the Wasserstein (a.k.a., Kantorovich-Rubinstein or earthmover) loss $d_{B_{p_d,q_d}^1}$ ($\sigma_d=1$ \citep{villani2008optimal}). Under these losses, the wavelet thresholding estimator is robustly minimax optimal, in both the arbitrary and structured contamination settings (note that here $\sigma_d\geq D/p_d=0$).
        In particular, in the case of unstructured contamination, this generalizes the results of \citet{chen2016general} for total variation loss to a range of other losses and smoothness assumptions on $p$.
        
        Analogously, in the case $p_g = q_g = \infty$, the data distribution is itself $\sigma_g$-H\"older continuous, since the Besov space $B_{p_g,q_g}^{\sigma_g} = \C^{\sigma_g}$ is equivalent to the space of $\sigma_g$-H\"older continuous functions. In this setting, the linear wavelet estimator is robustly minimax optimal under any Besov IPM loss when contamination is structured, or under sufficiently smooth Besov IPM losses (with $\sigma_d\geq D/p_d$) when the contamination is unstructured.

        One can also use our results to calculate the sensitivity of a given estimator to the proportion $\epsilon$ of outlier samples.
        In the terminology of robust statistics, this is quantified by the ``asymptotic breakdown point'' (i.e., the maximum proportion $\epsilon$ of outlier samples such that the estimator can still converge at the uncontaminated optimal rate). Figure~\ref{fig:breakdown_rates} illustrates the asymptotic breakdown point, in the case $p_d = 1$, as a function of the discriminator smoothness $\sigma_d$. For sufficiently smooth losses (large $\sigma_d$, the estimator can tolerate a large number ($O(\sqrt{n})$) of arbitrary outliers before performance begins to degrade, whereas, for stronger losses (smaller $\sigma_d$), the estimator becomes more sensitive to outliers.
    \begin{figure}
        \centering
        \begin{subfigure}[b]{0.65\textwidth}
            \centering
            \includegraphics[width=\textwidth]{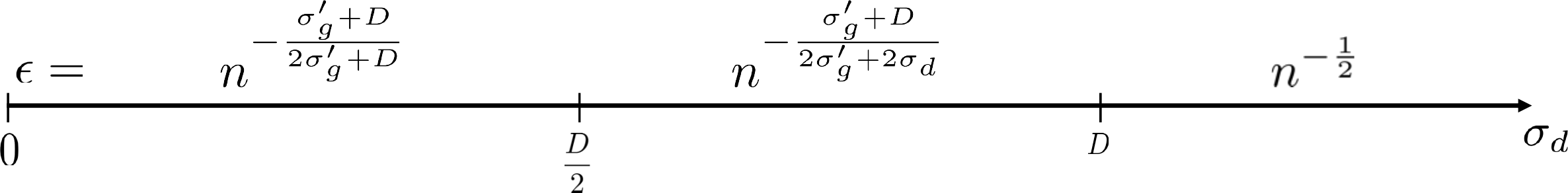}
            \vspace{8mm}
            \label{subfig:breakdown_rate}
        \end{subfigure}
        \hfill
        \begin{subfigure}[b]{0.32\textwidth}
            \centering
            \includegraphics[width=\textwidth]{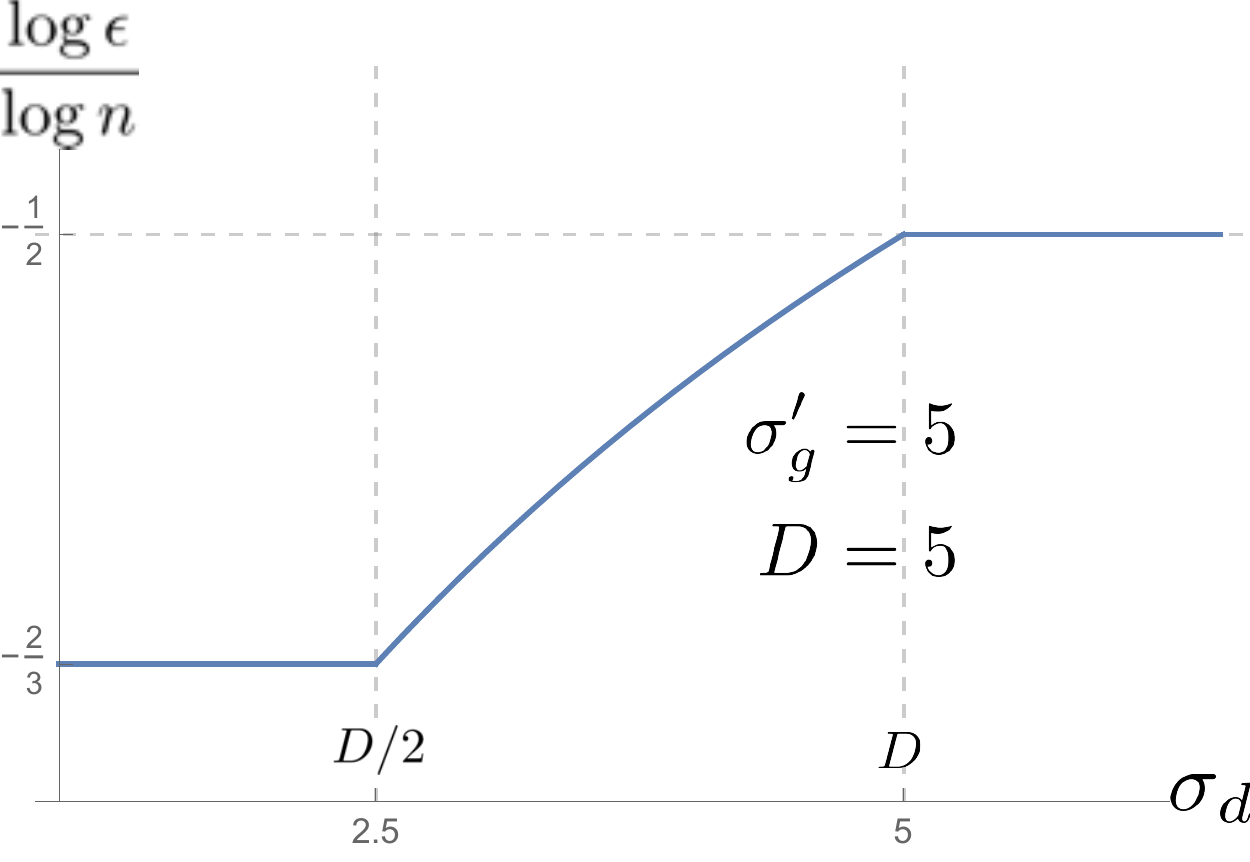}
            \label{subfig:plot_of_rate}
        \end{subfigure}
        \caption{Asymptotic breakdown point as a function of $\sigma_d$, in the case $p_d = 1$; this includes as special cases the $\L^\infty$ and Kolmogorov-Smirnov losses.}
        \label{fig:breakdown_rates}
    \end{figure}

\section{Robustness of Generative Adversarial Networks}
\label{sec:GANs}

    \citet{singh2018adversarial} showed (in their Theorem 9) that the problems of generating novel samples from a training density (also called ``implicit generative modeling''~\citep{mohamed2016learning}) and of estimating the training density are equivalent in terms of statistical minimax rates. Based on this result, and an oracle inequality of \citet{liang2018well}, several recent works \citep{liu2017approximationInGANs,liang2018well,singh2018adversarial,uppal2019nonparametric} have studied a statistical formulation of GANs as a distribution estimate based on empirical risk minimization (ERM) under an IPM loss. This formulation is as follows. Given a GAN with a discriminator neural network $N_d$ encoding functions in $\F$ and a generator neural network $N_g$ encoding distributions in $\P$, the GAN generator can be viewed as the distribution $\hat P$ satisfying:
    \begin{equation}
        \hat{P}  = \inf_{P\in \P} d_{\F}(P, \tilde{P}_n)
        \label{eq:GAN_estimate}
    \end{equation}
    While $\tilde{P}_n$ can be taken to be the empirical distribution $\frac{1}{n} \sum_{i = 1}^n \delta_{X_i}$, these theoretical works have shown that convergence rates can be improved by applying regularization (e.g., in the form of smoothing the empirical distribution), consistent with the ``instance noise trick'' \citep{sonderby2016amortised}, a technique that is popular in practical GAN training and is mathematically equivalent to kernel smoothing. Here, we extend these results to the contamination setting and show that the wavelet thresholding estimator can be used to construct a GAN estimate that is robustly minimax optimal.
    
    
    
    \citet{suzuki2018adaptivity}  (in section 3) showed that there is a fully connected ReLU network with depth at most logarithmic in $1/\delta$ and other size parameters at most polynomial in $1/\delta$ that can $\delta$-approximate any sufficiently smooth Besov function class (e.g. $B^{\sigma_g}_{p_g,q_g}$ with $\sigma_g\geq D/p_g$). This was used in \citet{uppal2019nonparametric} to show that, for large enough network sizes, the perfectly optimized GAN estimate (of the form of Eq.~\eqref{eq:GAN_estimate}) converges at the same rate as the estimator $\hat{p}$ used to generate it.
    So if we let the approximation error of the generator and discriminator network be at most the convergence rate (from Theorem \ref{thm:non_linear_unstructured_rate} or Theorem~\ref{thm:structured_rate}) of the wavelet thresholding estimator then there is a GAN estimate $\hat{P}$ that converges at the same rate and is therefore robustly minimax optimal.
    In particular, we have the following corollary:
    \begin{corollary}
        Given a Besov density class $B_{p_g,q_g}^{\sigma_g}$
    with $\sigma_g > D/p_g$ and discriminator class $B_{p_d,q_d}^{\sigma_d}$ with $\sigma_d>D/p_d$,
        there is a GAN estimate $\hat{P}$ with discriminator and generator networks of depth at most logarithmic in $n$, or $1/\delta$ and other size parameters at most polynomial in $n$ or $1/\delta$ such that
        \begin{align}
        \sup_{p\in \M\left( \epsilon, B^{\sigma_g}_{p_g,q_g} \right)} \E \left[ d_{B^{\sigma_d}_{p_d,q_d}}(\hat{p},p) \right]
        \leq
			\delta
			+ \mathcal{R} \left( n, B_{p_g,q_g}^{\sigma_g}, B_{p_d,q_d}^{\sigma_d} \right)
			+\epsilon
		\end{align}
    \end{corollary}
    
    Since Besov spaces of compactly supported densities are nested ($B^{\sigma}_{p,q}\subseteq B^{\sigma'}_{p,q}$ for all $\sigma'\leq \sigma$), to approximate $B^{\sigma}_{p,q}$ for any $\sigma\geq r_0$ it is sufficient to approximate $B^{r_0}_{p,q}$. We can use this approximation network along with the adaptive wavelet thresholding estimator to construct a GAN estimate of the form of Eq.~\eqref{eq:GAN_estimate}. Then under structured contamination this GAN estimate is minimax optimal for any density of smoothness $\sigma_g\in [r_0,r]$ and does not require explicit knowledge of $\sigma_g$. Thus, it is adaptive.

    
    
\section{Conclusion}
    In this paper, we studied a variant of nonparametric density estimation in which a proportion of the data are contaminated by random outliers. For this problem, we provided bounds on the risks of both linear and nonlinear wavelet estimators, as well as general minimax rates.
    The main conclusions of our study are as follows:
    \begin{enumerate}[wide,topsep=0pt,labelindent=0pt]
        \item
        The classical wavelet thresholding estimator originally proposed by \citet{donoho1996density}, which is widely known to be optimal for uncontaminated nonparametric density estimation, continues to be, in many settings, minimax optimal in the presence of contamination.
        \item Imposing a simple structural assumption, such as bounded contamination, can significantly alter how contamination affects estimation risk. At the same time, additional smoothness assumptions have no effect. This contrasts from the case of estimating a density at a point, as studied by \citet{liu2017density} where the minimax rates get better with smoothness of the contamination density.
        \item Linear estimators, exhibit optimal dependence on the contamination proportion, despite having sub-optimal risk with respect to the sample size. Hence, the difference between linear and nonlinear models diminishes in the presence of significant contamination.
        \item For sufficiently smooth density and discriminator class, a fully-connected GAN architecture with ReLU activations can learn the distribution of the training data at the optimal rate, both (a) in the presence of contamination and (b) when the true smoothness of the density is not known.
    \end{enumerate}
    
    Our results both extend recent results on nonparametric density estimation under IPM losses~\citep{liang2018well,singh2018adversarial,uppal2019nonparametric} to the contaminated and adaptive settings and expand the study of nonparametric density estimation under contamination~\citep{chen2016general,liu2017density} to Besov densities and IPM losses.

\section*{Broader Impact}

Since this work is of a theoretical nature, it is unlikely to disadvantage anyone or otherwise have significant negative consequences. One of the main contributions of this paper is to quantify the potential effects of misspecification biases on density estimation. Hence, the results in this paper may help researchers understand the potential effects of misspecification biases that can arise when invalid assumptions are made about the nature of the data generating process.

\begin{ack}
The authors thank anonymous reviewers for the feedback on improving this paper. The authors declare no competing interests. This work was supported by National Science Foundation award number DGE1745016, a grant from JPMorgan Chase Bank, and a grant from the Lockheed Martin Corporation.
\end{ack}

\bibliographystyle{plainnat}
\bibliography{ref}

\newpage
\appendix
  
\section{Set up}
    Besov spaces rely on the notion of an $r$-regular multi-resolution approximation (MRA) of $\L_2(\R^D)$. In particular, the father wavelet of the wavelet basis used to define Besov spaces generates an MRA of $\L_2(\R^D)$. 
    
    The goal of a MRA is to efficiently approximate spatially varying smoothness. \citet{hardle2012wavelets} explains it as a formulation that makes mathematically precise the intuitive idea of partitioning the domain and applying Fourier analysis to each piece. 
    
    Here we formally define an $r$-regular multi-resolution approximation. 
    \begin{definition}
        A \emph{multiresolution approximation (MRA)} of $\L^2(\R^D)$ is a nested sequence $\{V_j\}_{j\in \Z}$ of closed linear subspaces of $L^2(\R^D)$ such that:
        \begin{enumerate}[wide,noitemsep]
            \item 
                $\bigcap_{j=-\infty}^\infty V_j = \{0\}$, and $\bigcup_{j=-\infty}^\infty V_j$ is dense in  $\L^2(\R^D)$.
            \item 
                For every $f\in \L^2(\R^D)$ and $k\in \Z^D$, 
                    $f(x) \in V_0$ if and only if $f(x-k)\in V_0$.
            \item 
                For every $f\in \L^2(\R^D)$ and $j\in \Z$, $f(x) \in V_j$ if and only if $f(2x)\in V_{j+1}$
            \item 
                There is a ``father wavelet'' such that $\phi\in V_0$, $\{\phi(x-k) : k\in \Z^D\}$ is an orthonormal basis of $V_0 \subset \L^2(\R^D)$.
        \end{enumerate}
    \end{definition}
    Given a father wavelet that generates a multi-resolution approximation, there exist ``mother'' wavelets with the following properties.
    
    \begin{lemma}[\citep{meyer1992wavelets}, Section 3.9]
    Let $\{V_j\}_{j \in \Z}$ be an MRA of $\L^2(\R^D)$ with father wavelet $\phi$.
    , and let $W_j$ be the orthogonal complement of $V_j$ in $V_{j+1}$.
    Then, for $E = \{0,1\}^D\setminus (0,\dots,0)$, there exist ``mother wavelets'' $\{\psi_\epsilon\}_{\epsilon \in E}$ such that
        \begin{enumerate}[noitemsep,topsep=0pt]
            \item 
                $\psi_\epsilon$ is rapidly decreasing for every multi-index $\alpha$ with $|\alpha|\leq r$ and every $\epsilon\in E$.
            \item 
                 The set $\{\psi(x-k)\}_{\epsilon\in E,k\in \Z}$ is an orthonormal basis of $W_j$.
            \item 
                For all $\alpha$ with $|\alpha|\leq r$ and $\epsilon\in E$,
                $\int x^\alpha \psi_\epsilon(x)dx = 0$.
        \end{enumerate}
    Moreover, $\{2^{D j/2}\psi_\epsilon(2^jx-k) : \epsilon\in E, k\in \Z^D\} \cup \{2^{D j/2}\phi(2^jx-k) : k\in \Z^D\}$ is an orthonormal basis of $V_j \subseteq \L^2(\R^D)$.
    \end{lemma}
    The $r$-regularity of the mother wavelet as described part 3 of the above lemma determines the $r$-regularity of the wavelet basis. 
\section{Upper Bounds}

\subsection{Non-Linear Rate}
    In this section we provide proofs of the upper bounds stated above under both structured and unstructured contamination i.e. theorems \ref{thm:non_linear_unstructured_rate} and \ref{thm:structured_rate}. We will use a scaled version of the wavelet thresholding estimator to demonstrate these results. The proofs follow along the same lines as those of the uncontaminated version except the usual bias-variance trade-off now has an additional term; the misspecification error. 
    
    In particular, the bound on the bias remains unchanged. Moreover, we show that for resolutions small enough the variance can be bounded by the same term as before. This is straightforward for the variance of the linear terms but somewhat involved for that of the non-linear terms. So, we derive the bound for the non-linear terms at the very end. 
    
    There is a qualitative difference between the misspecification error under the structured and unstructured settings. When the contamination density is bounded, the misspecification error is simple bounded by the contamination proportion $\epsilon$; in the unstructured setting, this error depends on the number of terms considered in the estimator. 
    
    We now provide the formal proof.
    Let 
        \[
            \P = \{p:p\geq 0, \norm{p}_{\L^1}=1, \text{supp}(p)\subseteq [-T,T]\}
        \]
    denote the set of densities that are supported on the interval $[-T,T]$.
    We have assumed that our discriminator and generator classes are, respectively,
        \begin{align*}
            \F_d &= \{f : \norm{f}^{\sigma_d}_{p_d, q_d}\leq L_d\}\\
            \text{ and } \quad \F_g &= \{p : \norm{p}^{\sigma_g}_{p_g, q_g}\leq L_g\}\cap\P.
        \end{align*}
	Let for any density function $p$ 
	    \[
	        \begin{aligned}[c|r]
			        \alpha^p_{\phi} &=  \E_{X\sim p}[\phi(X)]\\
			        \text{ and } \quad \beta^p_{\psi} &= \E_{X\sim p}[\psi(X)].
			    \end{aligned}
	    \]
    Since $p\in \F_g$, we have that
	\begin{align*}
	    p&= \sum_{ \phi\in \Phi} \alpha^p_{\phi} \phi+
		    \sum_{j\geq 0} \sum_{\psi \in \Psi_j} \beta^p_{\psi}\psi,
	\end{align*}
	where the convergence is in the $L_p$ norm.
	
    For the unstructured setting we merely assume that the contamination density is compactly supported on $[-T,T]$. Under the structured contamination setting, we additionally assume that the contamination density $g$ is essentially bounded i.e. $\F_c = \L^\infty(L_c)$ (where $L_c$ is a uniform bound on the $\L_\infty$ norm of any $g\in \F_c$).
	
	We first show that it is enough to consider the ``sparse'' case (so called by \citet{donoho1996density}) characterized by $p_d'\geq p_g$ by the following lemma.

	\begin{lemma}
	    \label{lemma:nested_ipm}
	    For $p_d'\leq p_g$ and compactly supported densities $p, q\in \L_{p_g}\subseteq \L_{p_d'}$ we have that, 
	    \[
	        d_{\B^{\sigma_d}_{p_d,q_d}}(p,q) \leq d_{\B^{\sigma_d}_{p_g', q_d}}(p,q).
	    \]
	\end{lemma}
	\begin{proof}
	    Suppose $p,q$ are compactly supported on $[-T,T]$ then it is enough to show that 
	    \[
	        \B^{\sigma_d}_{p_d,q_d}(T)\subseteq \B^{\sigma_d}_{p_g',q_d}(T).
	    \]
	    Using the fact that $p_d\geq p_g'$, this is clear since, 
	    \[
	       2^{j(\sigma_d+D/2-D/p_g')}\norm{\beta_{j}}_{p_g'}
	        \leq 
	       2^{j(\sigma_d+D/2-D/p_d)}\norm{\beta_{j}}_{p_d}
	    \]
	    by using the simple fact that for a $2^{Dj}$-dimensional vector $x$, $\norm{x}_{p_g'}\leq 2^{Dj(1/p_g'-1/p_d)}\norm{x}_{p_d}$.
	\end{proof}
	
	Let $\hat{p}_n$ be the wavelet thresholding estimator of $p$ introduced by \cite{donoho1996density};
    \[
    \hat{p}_n 
		=   \sum_{ \phi\in \Phi} \hat{\alpha}_{\phi} \phi+
		    \sum_{j=0}^{j_0} \sum_{\psi\in \Psi_j} \hat{\beta}_{\psi}\psi+
		    \sum_{j= j_0}^{j_1}\sum_{\psi\in \Psi_j} \tilde{\beta}_{\psi} \psi
	\]
	where we threshold the higher resolution terms i.e. 
	\[
	    \begin{aligned}[c]
	        \alpha^p_{\phi} &=  \E_{X\sim p}[\phi(X)]\\
	        \beta^p_{\psi} &= \E_{X\sim p}[\psi(X)]
	    \end{aligned}\hspace{3ex}\qquad
	    \begin{aligned}[c]
	        \hat{\alpha}_{\phi} &= \frac{1}{n}\sum_{i=1}^n \phi(X_i)\\
	        \hat{\beta}_{\psi} &= \frac{1}{n}\sum_{i=1}^n \psi(X_i)\\
	        \tilde{\beta}_{\psi} &= \hat{\beta}_{\psi}\mathbf{1}_{\{\hat{\beta}_{\psi}>t\}}
	    \end{aligned}
	 \]
    with threshold $t = K\sqrt{j/n}$, where $K$ is a constant to be specified later, and 
	\begin{align*}
	    2^{j_0} &= \sqrt{n}^{\frac{1}{\sigma_g+D/2}}\\
	    2^{j_1} &= \sqrt{n}^{\frac{1}{\sigma_g+D/2-D/p_g}} \wedge \epsilon^{-\frac{1}{\sigma_g+D-D/p_g}}
	\end{align*}
	We will use a scaled version of this estimator i.e. $\frac{1}{1-\epsilon}\hat{p}_n$. 
	
	We decompose the risk of the above estimator as follows. At each resolution $\hat{\alpha}_\phi$ or $\hat{\beta}_\psi$ is an unbiased estimate of the co-efficient of the contaminated density $(1-\epsilon)p+\epsilon g$. So, by the triangle inequality, we can decompose the error as 
		\begin{align}
		\notag
	    \E d_{\F}&\left(\frac{\hat{p}_n}{1-\epsilon}, p\right)\\
            &\leq \frac{1}{1-\epsilon}\E d_{\F}
                \left(
                \sum_{ \phi\in \Phi} \hat{\alpha}_{\phi} \phi,
		        \sum_{ \phi\in \Phi} (\alpha^p_\phi+ \epsilon \alpha^g_\phi)\phi
                \right) \\
            &+ \frac{1}{1-\epsilon}d_{\F}\left(
		    \sum_{j=0}^{j_0} \sum_{\psi\in \Psi_j} \hat{\beta}_{\psi}\psi,
		    \sum_{j=0}^{j_0} \sum_{\psi\in \Psi_j} 
		            (\beta^p_{\psi}+ \epsilon \beta^g_\psi)\psi
		        \right)\\
            &+ \frac{1}{1-\epsilon}d_{\F}\left(
            \sum_{j= j_0}^{j_1}\sum_{\psi\in \Psi_j} \tilde{\beta}_{\psi} \psi
            , \sum_{j= j_0}^{j_1}\sum_{\psi\in \Psi_j} (\beta^p_{\psi}+\epsilon \beta^g_{\psi}) \psi
            \right)\\
            &+ \frac{1}{1-\epsilon} d_{\F}\left(
            \sum_{ \phi\in \Phi} \alpha^p_{\phi}\phi+\sum_{j=0}^{j_1} \sum_{\psi\in \Psi_j} 
		           \beta^p_{\psi}\psi
            , p
            \right)\\
            &+ \epsilon d_{\F}
                \left(
                \sum_{ \phi\in \Phi} 
                \alpha^g_{\phi} \phi
                +\sum_{j=0}^{j_1} \sum_{\psi\in \Psi_j} 
		           \beta^g_{\psi}\psi
		        ,0
		        \right)
		    \label{eq:bias_var_decomp}
    \end{align}
    where the first two terms constitute the error of the linear terms, the third term is the error of the non-linear terms, the fourth term is the bias and the last term is the misspecification error, respectively. 
	
	We will use the following upper bounds on the bias and variance of a linear wavelet estimator (when $j_0=j_1$ above) from Appendix C of \citet{uppal2019nonparametric}.
    
    First we see that under Besov IPMs, if the moments of the wavelet co-efficients of the density don't grow too fast with the resolution then the variance of the linear wavelet estimator can be conveniently bounded. 
    \begin{lemma}({\bf Variance})
        \label{lemma:variance}
        Let $X_1,\dots, X_n \sim p$ where $p$ is compactly supported and $\F_d = B^{\sigma_d}_{p_d,q_d}$. If   $\E_p|\psi(X)|^{p_d'}\leq c_{p_d'} 2^{Dj(p_d'/2-1)}$ for all $\psi \in \Psi_j$, then the variance of a linear wavelet estimator $\hat{p}_n$ with $j_0$ terms i.e. 
            \[
                \hat{p}_n 
	        =   \sum_{ \phi\in \Phi} \hat{\alpha}_{\phi} \phi+
	            \sum_{j=0}^{j_0} \sum_{\psi\in \Psi_j} \hat{\beta}_{\psi}\psi
            \]
        is bounded by 
            \[
                d_{\F_d}(\hat{p}_n,\E[\hat{p}_n]) \leq c
                \left(     \frac{1}{\sqrt{n}}+\frac{2^{j_0(D/2-\sigma_d)}}{\sqrt{n}}
                \right)
            \]
        where $c = c_{p_d'} \left(\E_p|\psi(X)|^2\right)^{1/2}$ is a constant.
    \end{lemma}
    Note here that we do not need the density to lie in a Besov space but to simply have the given bound on the moments of its wavelet coefficients.
    However, for a bound on the bias provided below we need the full power of the Besov space.        
    \begin{lemma}({\bf Bias})
        \label{lemma:bias}
        Let $X_1,\dots, X_n \sim p$ where $p \in B^{\sigma_g}_{p_g,q_g}$ is compactly supported and $\sigma_g\geq D/p_g$, $\F_d = B^{\sigma_d}_{p_d,q_d}$. Then the bias of a linear wavelet estimator $\hat{p}$ with $j_0$ terms is bounded by 
            \[
                d_{\F_d}(p,\E_p[\hat{p}_n])
                \leq c 2^{-j_0(\sigma_d+\sigma_g-(D/p_g-D/p_d')_+)}
            \]
        where $c = L_d L_g$ is a constant.
        \end{lemma}

     We will also need the following bound on a density living in a Besov space.
     \begin{lemma}{\bf (Upper Bound on Smooth Besov Spaces)}
        \label{lemma:besov_ub}
        Let $f\in B^{\sigma_g}_{p_g,q_g}$ where $\sigma_g> D/p_g$ then 
        \[
            \norm{f}_\infty \leq 4A\norm{\psi}_\infty L_g (1-2^{(\sigma_g-D/p_g)q_g'})^{-1/q_g'}
        \]
    \end{lemma}
    This lemma implies that sufficiently smooth Besov spaces $B^{\sigma_g}_{p_g,q_q}$ are uniformly bounded.
    
    We are now ready to provide upper bounds on the risk in both the structured and unstructured setting whenever $p_d'\geq p_g$. 
    
    Under both the structured and unstructured contamination setting we can immediately bound the bias term using lemma~\ref{lemma:bias} (since this is just the bias of a linear wavelet estimator with $j_1$ terms) by 
    	\[
    	    2^{-j_1(\sigma_g+\sigma_d+D/p_d'-D/p_g)}
    	\]
    	
    Since, $\sigma_g> D/p_g$  we know that by lemma \ref{lemma:besov_ub}, $\norm{p}_\infty<\infty$. Therefore, for any $\psi\in \Psi_j$,
    	\begin{align*}
    	    \E_{(1-\epsilon)p+\epsilon g} \left[|\psi(X)|^{p_d'}\right]
    	    &\leq (1-\epsilon)\norm{p}_\infty2^{Dj(p_d'/2-1)} + \epsilon\E_{g} \left[|\psi(X)|^{p_d'}\right]
    	\end{align*}
   	
    When contamination is structured i.e. $\norm{g}_\infty<\infty$ we have 
    \[
        \epsilon\E_{g} \left[|\psi(X)|^{p_d'}\right] \leq \epsilon\norm{g}_\infty2^{Dj(p_d'/2-1)}
    \]
    and when the contamination density is not bounded above we have, 
	\[
	    \epsilon\E_g|\psi(X)|^{p_d'}
	    \leq \epsilon 2^{Djp_d'/2}
	\]
	Since in this case, $2^{Dj_1}\leq \epsilon^{-\frac{D}{\sigma_g+D-D/p_g}}\leq 1/\epsilon$, $\epsilon\leq 2^{-Dj}$ for all $j\leq j_1$ the term above is always smaller than $2^{Dj(p_d'/2-1)}$.
	
	So, under both cases, we have, 
    \begin{align}
    \label{eq:moment_bound}
        \E_{(1-\epsilon)p+\epsilon g} |\psi(X)|^{p_d'}
        \leq c2^{Dj(p_d'/2-1)}
    \end{align}
    
	We can now use lemma \ref{lemma:variance} to bound the variance for both cases as 
	    \[
	         \E d_{\F}(\hat{p}_n, \E_{(1-\epsilon)p+\epsilon g}[\hat{p}_n]) \leq c\left( \frac{1}{\sqrt{n}}+\frac{2^{j_0(D/2-\sigma_d)}}{\sqrt{n}} 
	         \right)
	    \]
    We can also similarly bound the mis-specification error as this is simply the misspecification error of a linear wavelet estimator with $j_1$ terms. We then have an upper bound of
     \[
         \epsilon 2^{j_1(D/p_d-\sigma_d)}.
    \]

      We now bound the misspecification error i.e.
	    \[
	        \frac{\epsilon}{1-\epsilon} d_{\F}\left(\E_{g}[\hat{p}_n], 0\right)
	    \]
    We will use the following lemmas proven by \citet{uppal2019nonparametric} to first reduce the expression of the above distance to one in terms of wavelet coefficients (of $g$) only.
    
     \begin{lemma}
        \label{lemma:coeff}
        Let $p$, $q$ be compactly supported probability densities and $\F_d = B^{\sigma_d}_{p_d,q_d}$, s.t. either $p, q\in L_{p_d'}$ or $\sigma_d>D/p_d$, then $d_{\F_d}(p,q)=$
        \[
            \sup_{f\in \F_d}\left| 
			\sum_{\phi\in \Phi}
			\alpha^f_{\phi}
			\left(\alpha_{ \phi}^p- \alpha_{\phi}^q\right) + 
			\sum_{j\geq 0}\sum_{\psi\in \Psi_j} 
			\beta^f_{\psi}
			\left( \beta_{\psi}^p- \beta_{\psi}^q\right)\right|
        \]
        where for $f\in \F_d$
            \[
    	    f = \sum_{\phi\in \Phi}\alpha^f_{\phi}\phi + 
    			\sum_{j\geq 0}\sum_{\psi\in \Psi_j} \beta^f_{\psi}\psi
    	    \]
    \end{lemma}
    
    \begin{lemma}
        \label{lemma:seq_norm}
	    Let $n_1, n_2\in \N \cup \{\infty\}$ and $\eta$ be any sequence of numbers. Then
	    \begin{align*}
	        \E_{X_1,\dots, X_n} &\sup_{f\in \F_d}\sum_{j=n_1}^{n_2}\sum_{\psi\in \Psi_j} \gamma_{\psi}^f\eta_{\psi}\leq 
	            L_d
		    \sum_{j=n_1}^{n_2}2^{-j\sigma_d'}
		    \left(\E_{X_1,\dots, X_n}
		    \sum_{\psi\in \Psi_j}|\eta_{\psi}|^{p_d'}
		    \right)^{1/p_d'}
	    \end{align*}
	    where $\sigma_d' = \sigma_d+D/2-D/p_d$.
	    Note that the above is true also if $\gamma = \alpha^f$ and $n_1=n_2=0$.
	\end{lemma}
		
	Applying the lemmas above we have for any contamination density $g$,
    \begin{align}
        \label{eq:ub_misspec}
        \epsilon d_{\F}\left(\E_{g}[\hat{p}_n], 0\right)\leq
        &c \epsilon
            \left(\norm{\alpha^g}_{p_d'} 
                +\sum_{j=0}^{j_1}2^{-j(\sigma_d+D/2-D/p_d)}\norm{\beta^g_j}_{p_d'}
           \right)
    \end{align}
        
    where for all $\phi\in \Phi$ and $\psi\in \Psi_j$, $\alpha^g_{\phi} = \int \phi(x)g(x)$ and $\beta^g_{\psi}= \int \psi(x)g(x)$.

	 When the contamination is structured we have the following upper bound on the wavelet coefficients 
	    \begin{align}
	        \label{eq:ub_struct_misspec}
            |\beta^g_{\psi} |
                = \left|\int \psi(x)g(x)d(x)\right|\leq \norm{\psi_\epsilon}_\infty \norm{g}_{\infty} 2^{-Dj/2}
                \implies \norm{\beta^g_\psi}_{p_d'} \leq c2^{Dj(1/p_d'-1/2)}
        \end{align}
        where $\psi \in \Psi_j$ and $|\alpha_\phi^g|\leq \norm{\phi}_\infty\norm{g}_\infty$. Thus implying the following bound on \ref{eq:ub_misspec}
        \begin{align*}
         c \epsilon
	        \left( 1 + \sum_{j=0}^{j_0}2^{-j(\sigma_d+D/2-D/p_d)}2^{Dj/p_d'}2^{-Dj/2}
	        \right)
	       \leq c \epsilon
	        \left( 1 + \sum_{j=0}^{j_0}2^{-j\sigma_d}
	        \right)
	       \leq c \epsilon.
        \end{align*}
    
    When the contamination is unstructured, by convexity we have, 
        \begin{align*}
         \norm{\beta^g_j}_{p_d'}
             = \left(
                \sum_{\psi\in \Psi_j} |\E_g \psi(X)|^{p_d'}
                \right)^{1/p_d'}
            \leq 
                 \left(
                \sum_{\psi\in \Psi_j} \E_g |\psi(X)|^{p_d'}
                \right)^{1/p_d'}
            \leq 
                \left(
                 \E_g \sum_{\psi\in \Psi_j}|\psi(X)|^{p_d'}
                \right)^{1/p_d'}
        \end{align*}
    and we can interchange the expectation and sum in the last step because $g$ is compactly supported which implies there are only finitely many non-zero terms to sum. The compactness of the wavelets implies only finitely many wavelets overlap at a point. So we have, 
        \begin{align}
            \int \left(
            \sum_{\psi\in \Psi_j} |\psi(x)|^{p_d'}
            \right)
            g(x)dx 
            \leq c 2^{Djp_d'/2}
            \implies \norm{\beta^g_j}_{p_d'}
            \leq c2^{Dj/2}
            \label{eq:ub_unstruct_misspec}
        \end{align}
    where $c$ might depend on the dimension. 
    So we obtain the bound, (where the $\alpha$ term is bounded in the same way by a constant)
        \begin{align*}
              \epsilon d_{\F}\left(\E_{G}[\hat{p}_n], 0\right)
              &\leq c \epsilon \left(1+
                \sum_{j=0}^{j_0}2^{-j(\sigma_d+D/2-D/p_d)}2^{Dj/2}
            \right)
            = c\epsilon\left(1+
                \sum_{j=0}^{j_0}2^{j(D/p_d-\sigma_d)}
            \right)\\
            &\leq 
            c\epsilon \left(1 + 2^{j_0(D/p_d-\sigma_d)}
                \right)
        \end{align*}
        
    So, it only remains to bound the risk of the non-linear terms i.e. 
    \[
        d_{\F}\left(
            \sum_{j= j_0}^{j_1}\sum_{\psi\in \Psi_j} \tilde{\beta}_{\psi} \psi
            , \sum_{j= j_0}^{j_1}\sum_{\psi\in \Psi_j} (\beta^p_{\psi}+\epsilon\beta^g_{\psi}) \psi 
            \right)
    \]
    From lemmas \ref{lemma:coeff} and \ref{lemma:seq_norm} we will upper bound the following: 
    \[
        \sum_{j= j_0}^{j_1} 2^{-j(\sigma_d+D/2-D/p_d)}
            \left( \sum_{\psi \in\Psi_j}\E
		|(1-\epsilon)\beta^p_{\psi}+\epsilon\beta^g_{\psi}- \tilde{\beta}_{\psi}|^{p_d'}
        \right)^{1/p_d'}    
    \]
    We will need the following moment and large deviation bounds from \cite{uppal2019nonparametric}:
    
    \begin{lemma}(Moment Bounds)
        Let $X_1, \dots, X_n \sim p$, $m\geq 1$ s.t. there is a constant $c$ with $\E_p|\psi(X)|^m\leq c 2^{Dj(m/2-1)}$ for all $\psi\in \Psi_j$. Let
            \begin{align*}
				\gamma_{\psi}^p 
				    &= \E[\psi(X)],\\
				\hat{\gamma}_{\psi} 
				    &= \frac{1}{n}\sum_{i=1}^n \psi(X_i),
			\end{align*}
	    Then for all $j$ s.t. $2^{Dj}\in \mathcal{O}(n)$, 
			\[
			    \E[|\hat{\gamma}_{\psi}-\gamma_{\psi}|^m]
			        \leq c n^{-m/2}.
			\]
        where $c = c_m \left(\E_p|\psi(X)|^2\right)^{m/2}$ is a constant.
    \end{lemma}
    
    \begin{lemma}(Large Deviations)
        Let $X_1, \dots, X_n \sim p$ such that for a constant $c$, $\E_p|\psi(X)|^2\leq c $ for $\psi\in \Psi_j$. Let
            \begin{align*}
				\gamma_{\psi}^p 
				    &= \E[\psi(X)],\\
				\hat{\gamma}_{\psi} 
				    &= \frac{1}{n}\sum_{i=1}^n \psi(X_i),
			\end{align*}
        Let $l=\sqrt{j/n}$ and $\gamma>0$, then, for all $j$ s.t. $2^{Dj}\in o(n)$,
		    we have,
		    \begin{align*}
		    \Pr(|\hat{\gamma}_{\psi}-\gamma_{\psi}|>(K/2)l)
				&\leq 2 \times 2^{-\gamma n l^2}
		     \end{align*}
		where $K$ large enough such that 
	     \[
	            \frac{K^2}
				{8(c+\norm{\psi_\epsilon}_\infty(K/3))}> \log 2\gamma
	       \]
    \end{lemma}
    
    Both moment and large deviation bounds from above hold for all $j$ s.t. $\E_p|\psi(X)|^m\leq c 2^{Dj(m/2-1)}$ which we have shown to hold for all $j\leq j_1$ (see equation~\ref{eq:moment_bound}) under both cases.
    
    We now provide a general lemma bounding the non-linear term that we will also use when we provide a bound on the risk of the adaptive estimator. 
    
    \begin{lemma}
        \label{lemma:non_linear_terms}
        Let $X_1,\dots, X_n \sim p$ where $p$ is compactly supported and $\F_d = B^{\sigma_d}_{p_d,q_d}$. If   $\E_p|\psi(X)|^{p_d'}\leq c_{p_d'} 2^{Dj(p_d'/2-1)}$, then the risk of the non-linear terms of the wavelet thresholding estimator defined above i.e. 
            \[
                d_{\F}\left(
            \sum_{j= j_0}^{j_1}\sum_{\psi\in \Psi_j} \tilde{\beta}_{\psi} \psi
            , \sum_{j= j_0}^{j_1}\sum_{\psi\in \Psi_j} (\beta^p_{\psi}+\epsilon\beta^g_{\psi}) \psi 
            \right)
            \]
        is bounded by 
            \begin{align}
                \label{eq:ub_non_linear_var}
                  \sum_{j=j_0}^{j_1}  
    	        2^{-j(\sigma_d+D/2-D/p_d)} 
                \frac{\norm{\beta_j^p}_s^{s/p_d'}+
                \epsilon^{s/p_d'}\norm{ \beta_j^g}_{s}^{s/p_d'}}
                {\sqrt{n}^{1-s/p_d'}}.
            \end{align}
        for any $s\in [p_g, p_d']$.
    \end{lemma}
    
    \begin{proof}
        We follow the procedure of \cite{donoho1996density} and \cite{uppal2019nonparametric} and break up the term into different cases. The first two of which correspond to the situation where the empirical estimate and the true value of the co-efficient are far apart. Similar to the uncontaminated case, using the large deviation bounds above we show that the probability of this happening is negligible.
        
        This leaves us with two cases to consider: when the estimate $\hat{\beta}_\psi$ and the true coefficient are either both small or both large. We show that both of these cases reduce to the same term which we then bound using the properties of Besov spaces and the compactness of all densities considered.
    
    	\begin{enumerate}
    	\item 
    	    Let $A$ be the set of $\psi\in \Psi_j$
    		s.t. $\hat{\beta}_{\psi}>t$ and $(1-\epsilon)\beta^p_{\psi}+\epsilon \beta^g_{\psi}<t/2$ and $r\geq 1/p_d'$ then by H\"older's inequality,
    		\begin{align*} 
    		    \sum_{j= j_0}^{j_1} 
    			&2^{-j(\sigma_d+D/2-D/p_d)}
    			\times
        		\left(
        			\sum_{\psi \in\Psi_j}
        			\E |\beta^p_{\psi}- \tilde{\beta}_{\psi}|^{p_d'}\mathrm{1}_A
        		\right)^{1/p_d'}
        		 \\
    			&\leq\sum_{j= j_0}^{j_1} 
    			2^{-j(\sigma_d+D/2-D/p_d)}
    			\times
    			\left(
    				\sum_{\psi \in\Psi_j}
    				(\E |\beta^p_{\psi}- \tilde{\beta}_{\psi}|^{p_d'r})^{1/r}\Pr(A)^{1/r'}
    			\right)^{1/p_d'}.
            \end{align*}
            Using the large deviation and moment bound we get an upper bound,
            \begin{align*}
    				\sum_{j= j_0}^{j_1} 
    				&c2^{-j(\sigma_d+D/2-D/p_d)}
    					\left( 
    						2^{D j} 
    							n^{-p_d'/2}
    							2^{-j\gamma/r'}
    					\right)^{1/p_d'}\\
    				&\leq 
    					c n^{-1/2}
    					2^{-j_0(\sigma_d-D/2+\gamma /p_d'r')}
    			\end{align*}
    		which is negligible compared to the linear term for large enough $\gamma$.
		
    		\item 
    		    Let $B$ be the set of $\psi\in \Psi_j$
    			s.t. $\hat{\beta}_{\psi}<t$ and $(1-\epsilon)\beta^p_{\psi}+\epsilon \beta^g_{\psi}>2t$ then same as above
    			\begin{align*}
    			    \sum_{j=j_0}^{j_1} 
    			    &2^{-j(\sigma_d+D/2-D/p_d)}
    				\norm{\beta_j^p+\epsilon\beta_j^g}_{p_d'}2^{-\gamma j/p_d'}\\
    				&\leq 2^{- j_0(\sigma_d+\sigma_g'+\gamma)}+ \epsilon\sum_{j=j_0}^{j_1} 2^{-j(\sigma_d+D/2-D/p_d)}2^{-\gamma j/p_d'}\norm{\beta_j^g}_{p_d'}
    	        \end{align*}
    		which is negligible compared to the bias term and the misspecification error for large enough $\gamma$. 
    
        In other words, for the upper bounds of the first two cases we have chosen $\gamma$ (which in turn determines the value of the constant $K$ for the threshold $t=K\sqrt{j/n}$) to be large enough so that the exponent of $2^j$ in the upper bound of these two terms is negative. This enables us to upper bound the geometric series (as a sum of $j$) by a constant multiple of the first term.
    	\item 
    	    Let $C$ be the set of $\psi\in \Psi_j$ s.t. $|\hat{\beta}_{\psi}|>t$ and $|(1-\epsilon)\beta^p_{\psi}+\epsilon \beta^g_{\psi}|>t/2$ then for any $p_g\leq s\leq p_d'$,
        	\begin{align*}
        	    \sum_{j= j_0}^{j_1} 
        			&2^{-j(\sigma_d+D/2-D/p_d)}\times\left( 
        				\sum_{\psi\in C}
        				\E|(1-\epsilon)\beta^p_{\psi}+\epsilon \beta^g_{\psi}- \tilde{\beta}_{\psi}|^{p_d'}
        			\right)^{1/p_d'}\\
        		&\leq 
        			\sum_{j= j_0}^{j_1} C 
        			2^{-j(\sigma_d+D/2-D/p_d)} \sqrt{j}^{s/p_d'}\times
        			\frac{\norm{(1-\epsilon)\beta^p_j+\epsilon\beta^g_j}_{s}^{s/p_d'}}{\sqrt{n}^{1-s/p_d'}}
        	\end{align*}
        	where we have used the moment bound and the lower bound on $(1-\epsilon)\beta^p_\psi+\epsilon\beta^g_{\psi}$.

    	\item 
    	    Let $E$ be the set of $\psi\in \Psi_j$
    		s.t. $|\hat{\beta}_{\psi}|<t$ and $|(1-\epsilon)\beta^p_\psi+\epsilon\beta^g_{\psi}|<2t$ then for any $p_g\leq s\leq p_d'$:
    		\begin{align*}
    				\sum_{j= j_0}^{j_1} 
    					&2^{-j(\sigma_d+D/2-D/p_d)}
    					\left(
    					\sum_{\psi\in E} |(1-\epsilon)\beta_{\psi}^p+\epsilon \beta^g_\psi|^{p_d'}\right)^{ 1/p_d'}\\
    				& \leq
    					\sum_{j= j_0}^{j_1} 
    					2^{-j(\sigma_d+D/2-D/p_d)} \times\left(
    					\sum_{\psi\in \Psi_j} |(1-\epsilon)\beta^p_\psi+\epsilon\beta^g_{\psi}|^{s}(2t)^{p_d'-s}\right)^{ 1/p_d'} \\
    				&= 
        			\sum_{j= j_0}^{j_1}
        			2^{-j(\sigma_d+D/2-D/p_d)} \sqrt{j}^{s/p_d'}\times
        			\frac{\norm{(1-\epsilon)\beta^p_j+\epsilon\beta^g_j}_{s}^{s/p_d'}}{\sqrt{n}^{1-s/p_d'}}
    		\end{align*}
    		where we have used the upper bound on $(1-\epsilon)\beta^p_\psi+\epsilon\beta^g_{\psi}$.
    	\end{enumerate}
    	
    	By applying Jensen's inequality we can show that both 3 and 4 above are bounded, for any $s\in [p_g, p_d']$, by the following (where we omit the $\sqrt{j}$ term since it only contributes a factor of polylog of $n$ or $\epsilon$ to the upper bound), 
    	\[
    	\sum_{j=j_0}^{j_1}  
    	        2^{-j(\sigma_d+D/2-D/p_d)} 
                \frac{\norm{\beta_j^p}_s^{s/p_d'}+
                \epsilon^{s/p_d'}\norm{ \beta_j^g}_{s}^{s/p_d'}}
                {\sqrt{n}^{1-s/p_d'}}.
    	\]
    \end{proof}
	We can now bound \ref{eq:ub_non_linear_var} by 
	    	\[
    	    \sum_{j=j_0}^{j_1}  
    	        2^{-j(\sigma_d+D/2-D/p_d)} 
                \frac{\norm{\beta_j^p}_{p_g}^{s/p_d'}+
                \epsilon^{s/p_d'}\norm{ \beta_j^g}_{p_g}^{s/p_d'}}
                {\sqrt{n}^{1-s/p_d'}}.
    	\]
    Let $A$ be the set of $j$ s.t. $\norm{\beta^p_j}_{p_g}\geq \epsilon\norm{\beta^g_j}_{p_g}$ and $B = [j_0,j_1]\setminus A$. Then the above is upper bounded by 
	 \begin{align*}
            \sum_{j\in A} 
	        &2^{-j(\sigma_d+D/2-D/p_d)}
         \frac{\norm{\beta_j^p}_{p_g}^{p_g/p_d'}}{\sqrt{n}^{1-p_g/p_d'}}
         + \epsilon\sum_{j\in B} 
	        2^{-j(\sigma_d+D/2-D/p_d)}\norm{ \beta_j^g}_{p_g}\\
	   &\leq 
	    \sum_{j= j_0}^{j_1} 
	        2^{-j(\sigma_d+D/2-D/p_d)}
         \frac{\norm{\beta_j^p}_{p_g}^{p_g/p_d'}}{\sqrt{n}^{1-p_g/p_d'}}
         + \epsilon\sum_{j= j_0}^{j_1}  
	        2^{-j(\sigma_d+D/2-D/p_d)}\norm{ \beta_j^g}_{p_g}\\
	   &\leq 
	     n^{1/2(p_g/p_d'-1)} 2^{-j_m((\sigma_g+D/2)p_g/p_d'+\sigma_d-D/2)}
	     + \epsilon\sum_{j= j_0}^{j_1}  
	        2^{-j(\sigma_d+D/2-D/p_d)}\norm{ \beta_j^g}_{p_g}
	 \end{align*}
	 where the second term is bounded by the misspecification error. 

    So for both the structured and unstructured contamination setting the bound on all terms except the misspecification error is the same. In particular, we have, for the structured setting, 
        \[ 
	    \frac{1}{\sqrt{n}}+\frac{2^{j_0(D/2-\sigma_d)}}{\sqrt{n}}  +2^{-j_0(\sigma_d+\sigma_g-D/p_g+D/p_d')}+ n^{1/2(p_g/p_d'-1)} 2^{-j_m((\sigma_g+D/2)p_g/p_d'+\sigma_d-D/2)}+\epsilon
	    \]
	    which gives, 
	    \[
				\frac{1}{\sqrt{n}}+	n^{-\frac{\sigma_g+\sigma_d}{2\sigma_g+D}}+ n^{-\frac{\sigma_g+\sigma_d-D/p_g+D/p_d'}{2\sigma_g+D-2D/p_g}}
				+\epsilon
	    \]
    In contrast for the unstructured setting we have, 
        \begin{align*}
            &\frac{1}{\sqrt{n}}+\frac{2^{j_0(D/2-\sigma_d)}}{\sqrt{n}}  +2^{-j_1(\sigma_d+\sigma_g-D/p_g+D/p_d')}+ n^{1/2(p_g/p_d'-1)} 2^{-j_m((\sigma_g+D/2)p_g/p_d'+\sigma_d-D/2)}\\
            &\epsilon2^{Dj_1(D/p_d-\sigma_d)}+\epsilon
        \end{align*}

    At the given values of $j_0$ and $j_1$ this gives us an upper bound of 
    \begin{align*}
         &\frac{1}{\sqrt{n}} + \sqrt{n}^{-\frac{\sigma_g+\sigma_d}{\sigma_g+D/2}}+
        \sqrt{n}^{-\frac{\sigma_g+\sigma_d+D/p_d'-D/p_g}{\sigma_g+D/2-D/p_g}}+\\
        &
        \epsilon + \epsilon ^{\frac{\sigma_g+\sigma_d+D/p_d'-D/p_g}{\sigma_g+D-D/p_g}}.
    \end{align*}
	
	We note that the above proof implicitly assumes that $p_d'<\infty $ or equivalently $p_d>1$. We provide a bound for the case $p_d'=\infty$ in the next section since in this case it is sufficient to look at linear estimators. 
	
\subsection{Linear Rate}
    In this section we provide an upper bound on the risk of the linear wavelet estimator which is simply a non-linear estimator with the added constraint that $j_0=j_1$ or without its non-linear terms. Again, in view of lemma~\ref{lemma:nested_ipm} it is sufficient to consider the case $p_d'\geq p_g$. 
    
    We use the upper bounds on the components of the error of the non-linear wavelet estimator computed above with the additional constraint that $j_0=j_1$. Therefore we have the following upper bounds along with the implied rate. 
    In the unstructured setting, the upper bound is 
        \begin{align*}
            \frac{1}{\sqrt{n}} &+\frac{2^{j_0(D/2-\sigma_d)}}{\sqrt{n}}  +2^{-j_0(\sigma_d+\sigma_g-D/p_g+D/p_d')}\\
            \epsilon &+ \epsilon2^{Dj_0(D/p_d-\sigma_d)}
        \end{align*}
        which implies the rate
        \begin{align*}
         &\frac{1}{\sqrt{n}} +
        n^{-\frac{\sigma_g+\sigma_d+D/p_d'-D/p_g}{2\sigma_g+D-2D/p_g+2D/p_d'}}+
        \epsilon + \epsilon ^{\frac{\sigma_g+\sigma_d+D/p_d'-D/p_g}{\sigma_g+D-D/p_g}}.
    \end{align*}
    In the structured setting, the upper bound is,
        \begin{align*}
            \frac{1}{\sqrt{n}} +\frac{2^{j_0(D/2-\sigma_d)}}{\sqrt{n}}  +2^{-j_0(\sigma_d+\sigma_g-D/p_g+D/p_d')}+
            \epsilon 
        \end{align*}
        which implies the rate
        \begin{align*}
             &\frac{1}{\sqrt{n}} +
            n^{-\frac{\sigma_g+\sigma_d+D/p_d'-D/p_g}{2\sigma_g+D-2D/p_g+2D/p_d'}}+
            \epsilon.
        \end{align*}
    For $p_d'=\infty$ the bounds on the bias or the misspecification error still hold i.e. 
    \[
        2^{-j_0(\sigma_g+\sigma_d-D/p_g)} + \epsilon + \epsilon 2^{j_0(D-\sigma_d)}
    \]
    The variance bound is given by 
    \[
           \sum_{j=0}^{j_0}2^{-j(\sigma_d+D/2-D/p_d)}\E\sup_{\psi\in \Psi_j} |\hat{\beta}_\psi-\beta^p_\psi| \leq \frac{2^{j(D/2-\sigma_d)}}{\sqrt{n}}
    \]
    This is shown using the lemma above bounding large deviations. We have, 
    \begin{align*}
        \P\left(\sup_{\psi\in \Psi_j} |\hat{\beta}_\psi-\beta^p_\psi|\geq K\sqrt{j/n}\right)
            \leq 2^{j(D-\gamma)}
    \end{align*}
    which implies 
    \[
        \sum_{j=0}^{j_0}\frac{ 2^{j(2D-\sigma_d-\gamma)}}{\sqrt{n}}\leq \frac{1}{\sqrt{n}}
    \]
    for $\gamma$ sufficiently large.

    Now, with a choice of $2^{j_0}=n^{\frac{1}{2\sigma_g+D-2D/p_g}}\wedge \epsilon^{-\frac{1}{\sigma_g+D-D/p_g}}$ we have an upper bound of
    \[
        \frac{1}{\sqrt{n}}+n^{-\frac{\sigma_g+\sigma_d-D/p_g}{2\sigma_g+D-2D/p_g}} + \epsilon + \epsilon^{\frac{\sigma_g+\sigma_d-D/p_g}{\sigma_g+D-D/p_g}}
    \]
    Similarly, when contamination is structured we have a bound of 
    \[
    \frac{1}{\sqrt{n}}+n^{-\frac{\sigma_g+\sigma_d-D/p_g}{2\sigma_g+D-2D/p_g}} + \epsilon.
    \]
\subsection{Dense case: $\boldsymbol{p_d'\leq p_g}$}
    Here we provide a better upper bound on the risk when $p_d'\leq p_g$ using the linear estimator from above. In this case we obtain a better bound without using the monotonicity of the dual Besov norms from lemma 9. While most of the components of the proof are the same as in the non-linear case of Section B.1 the bound on the variance is a little more involved. 
    
    \begin{proof}
	Given $X_1, \dots, X_n\IID (1-\epsilon)p+\epsilon g$. 
    Let our estimator be $\hat{p}=\frac{1}{1-\epsilon}\hat{p}_n$ where $\hat{p}_n$ is the linear wavelet estimator defined above with 
	\begin{align*}
	    2^{j_0} &= n^{\frac{1}{2\sigma_g+D}}\wedge \epsilon^{-\frac{1}{\sigma_g+D/p_d}}
	\end{align*}

	We use the same bias variance decomposition as in the proof of the non-linear estimator with the additional constraint that $j_0=j_1$. Therefore we have no non-linear terms to bound. We have the following upper bound on the error.  
	\begin{align*}
	    \E d_{\F_d}(\hat{p},p) 
	        &\leq \frac{1}{1-\epsilon}\E d_{\F}(\hat{p}_n, \E_{(1-\epsilon)p+\epsilon g}[\hat{p}_n]) \\
            &+ d_{\F}\left(\E_{p}[\hat{p}_n], p\right)
            + \frac{\epsilon}{1-\epsilon} d_{\F}\left(\E_{g}[\hat{p}_n], 0\right)
	\end{align*}
	where the first term is the stochastic error or the variance, the second term is the bias and the third term is the misspecification error. Here again, the bound on the bias is unchanged. Moreover, the misspecification error can be bounded in the same way as in the proof of the sparse case above. We can also show here that the variance bound remains the same but this is not as straightforward since we don't just have lower resolution terms.
	
	Using lemma ~\ref{lemma:bias} we get the same bound as before on the bias or the second term above i.e. 
	    \[
	        d_{\F}\left(\E_{p}[\hat{p}_n], p\right) \leq c 2^{-j_0(\sigma_d+\sigma_g-(D/p_g-D/p_d')_+)}
	    \]
	
    Now we bound the variance or the first term. Let $\psi\in \Psi_j$ and 
	\begin{align*}
	        Y_i 
		    &= \psi_(X_i)- \E[\psi(X)]
	\end{align*}
	then for all $m\geq 1$, applying first the triangle inequality and then Jensen's inequality repeatedly we get
		\begin{align*}
			\E[|Y_i|^m] 
			    &\leq \E[\left(|\psi(X_i)|+|\E[\psi(X_i)]|\right)^m] \\
			    &\leq 2^{m-1}\left( \E[|\psi(X_i)|^m]+|\E[\psi(X_i)]|^m
			        \right) 
			        \\
				&\leq 
					2^{m}\E[|\psi(X_i)|^m]. 
		\end{align*}
        
    Therefore, by Rosenthal's inequality, i.e., 
        \begin{lemma}({\bf Rosenthal's Inequality} (\cite{rosenthal1970subspaces}))
        Let $m\in \R$ and $Y_1,\dots, Y_n$ be IID random variables with $\E[Y_i]=0$, $\E[Y_i^2]\leq \sigma^2$. Then there is a constant $c_m$ that depends only on $m$ s.t. 
	        \begin{align*}
	            \E\left[
	                \left|\frac{1}{n}\sum_{i=1}^n Y_i
	                \right|^m
	               \right] 
	                \leq c_m
	                    \left(
	                        \frac{\sigma^m}{n^{m/2}} + \frac{\E|Y_1|^m}{n^{m-1}}1_{2<m<\infty}
	                    \right)
	        \end{align*}
    
    \end{lemma}
    we have,
		\begin{align*}
            &\E[|\hat{\beta}_{\psi}-(1-\epsilon)\beta_\psi^p-\epsilon\beta_\psi^g|^{p_d'}]\\
				&\leq 
		         c_{p_d'} \left( \left(\E|\psi(X)|^2\right)^{p_d'/2}
		          + 
		          \frac{\E[|\psi(X)|^{p_d'}]}{n^{p_d'-1}}1_{p_d'\geq 2}
					\right)
		\end{align*}

	where $c_{p_d'}$ is a constant that only depends on $p_d'$. 

    This implies that the variance is bounded by:
    \begin{align*}
	    &
		 \sum_{j=0}^{j_0}2^{-j(\sigma_d+D/2-D/p_d)}\times\\
	   &\left(
	        \sum_{\psi\in \Psi_j} \E|\hat{\beta}_{\psi}-(1-\epsilon)\beta^p_\psi -\epsilon\beta^g_{\psi}|^{p_d'}
	    \right)^{1/p_d'}
				\\
		\leq 
		 &\sum_{j=0}^{j_0}\frac{2^{-j(\sigma_d+D/2-D/p_d)}}{\sqrt{n}}\times \\
		 &\left(
	        \sum_{\psi\in \Psi_j} \left(\E|\psi(X)|^2\right)^{p_d'/2}+ \frac{\E[|\psi(X_i)|^{p_d'}]}{n^{p_d'/2-1}}1_{p_d'\geq 2}
		 \right)^{1/p_d'}
	\end{align*}
	Now we can bound each of the terms inside the brackets separately. The second term is bounded as
	\begin{align*}
	    &\sum_{\psi\in \Psi_j}
	    \frac{\E[|\psi(X_i)|^{p_d'}]}{n^{(p_d'/2-1)}}1_{p_d'\geq 2}\\
        &\leq\sum_{\psi\in \Psi_j} 
         \frac{(1-\epsilon)\E_p[|\psi(X_i)|^{p_d'}]+\epsilon\E_g[|\psi(X_i)|^{p_d'}]}{n^{(p_d'/2-1)}}1_{p_d'\geq 2}\\
	    &\leq 
        \sum_{\psi\in \Psi_j}  \frac{2^{Dj}2^{Dj(p_d'/2-1)}  + \epsilon 2^{Djp_d'/2}}{n^{(p_d'/2-1)}}1_{p_d'\geq 2}
	    \\
	    &\leq
        \sum_{\psi\in \Psi_j} \frac{2^{Djp_d'/2} }{n^{(p_d'/2-1)}}1_{p_d'\geq 2} \leq 2^{Dj}1_{p_d'\geq 2}
	\end{align*}
    While the first term is bounded as
	\begin{align*}
	    &\sum_{\psi\in \Psi_j} \left(\E|\psi(X)|^2\right)^{p_d'/2}\\
	    &\leq
	    \sum_{\psi\in \Psi_j} \left((1-\epsilon)\E_p|\psi(X)|^2+\epsilon\E_g|\psi(X)|^2\right)^{p_d'/2}\\
	    &\leq 
	        \sum_{\psi\in \Psi_j} 
	        \left((1-\epsilon)\norm{p}_\infty + \epsilon 2^{Dj}w_\psi\right)^{p_d'/2}
	\end{align*}
	where $w_\psi$ is $\int 1_{\text{supp}(\psi)}g(x)dx$. Since we know that at any point at most finitely many wavelets intersect $\sum w_\psi \leq c \int g(x)dx =c$.
	    
	For $p_d'\leq 2$ by Jensen's we have, 
    \begin{align*}
         &2^{Dj} \left(
         \frac{1}{2^{Dj}} \sum_{\psi\in \Psi_j}(1-\epsilon)\norm{p}_\infty + \epsilon 2^{Dj}w_\psi\right)^{p_d'/2} \\
        &\leq 2^{Dj} \left(
         c + \epsilon \right)^{p_d'/2} \leq 2^{Dj}
    \end{align*}
	    
	For $p_d'\geq 2$ again by Jensen's, we have,
	    \begin{align*}
	        (1-\epsilon)2^{Dj} + 
	        &(\epsilon 2^{Dj})^{p_d'/2} \norm{w}_{p_d'/2}^{p_d'/2}\\
	        &\leq 
	            2^{Dj} + (\epsilon 2^{Dj})^{p_d'/2}
	    \end{align*}

    where we have used the fact that $\norm{w}_{p_d'/2}\leq \norm{w}_1$. Since $2^{Dj_0}\leq (1/\epsilon)^{\frac{D}{\sigma_g+D/p_d}}$, for every $j\leq j_0$, $\epsilon\leq 2^{-j(\sigma_g+D/p_d)}$. This implies 
	    
    \begin{align*}
        (\epsilon 2^{Dj})^{p_d'/2} \leq 2^{j(D/p_d'-\sigma_g)p_d'/2} 
            =2^{j(D/2-\sigma_gp_d'/2)} \leq 2^{Dj}
    \end{align*}
    In conclusion, the sum of the variance terms at any resolution $j$ (not too large) is bounded by $2^{Dj/p_d'}$.
    Therefore, we have an upper bound for the variance term, which is the same as usual, i.e., 
	\begin{align*}
		&\sum_{j=0}^{j_0}2^{-j(\sigma_d+D/2-D/p_d)}n^{-1/2} 
		    2^{D j/p_d'}\\
	 	&\leq  
			\sum_{j=0}^{j_0}2^{j(D/2-\sigma_d)}
			n^{-1/2}\\
		&\leq 
            \frac{1}{\sqrt{n}} 
            +\frac{2^{j_0(D/2-\sigma_d)}}{\sqrt{n}} 
	\end{align*}

	It only remains to bound the last term, or the misspecification error, which we can bound in the same as the non-linear case of section B.1 i.e. we have, 
	    \[
	        \frac{\epsilon}{1-\epsilon} d_{\F}\left(\E_{g}[\hat{p}_n], 0\right) \leq 
	        c\epsilon \left(1 + 2^{j_0(D/p_d-\sigma_d)}
                \right)
	    \]
	Therefore, our upper bound is $\lesssim$
	    \begin{align*}
			&\frac{1}{\sqrt{n}}
			+n^{-\frac{\sigma_g+\sigma_d}{2\sigma_g+D}}
			+n^{-\frac{\sigma_g+\sigma_d+D/p_d'-D/p_g}{2\sigma_g-2D/p_g+2D/p_d'+D}}+\\
			&\epsilon
			+\epsilon^{\frac{\sigma_g+\sigma_d}{\sigma_g+D/p_d}}
			+\epsilon^{\frac{\sigma_g+\sigma_d+D/p_d'-D/p_g}{\sigma_g-D/p_g+D}}
    	\end{align*}
	
	\end{proof}

\subsection{Adaptivity}
    We now provide a version of the thresholding wavelet estimator above that is, under the structured contamination setting, adaptive to both the contamination proportion $\epsilon$ and smoothness of the true density $\sigma$. This essentially follows from the argument provided by \cite{donoho1996density} except that we extend it to higher dimensions. We reproduce the proof here for completeness
    
    Given lemma \ref{lemma:nested_ipm} we only consider the case $p_d'\geq p_g$.
    
    We now construct the adaptive version of the thresholding wavelet estimator.
    
    Firstly, we no longer use a scaled version of $\hat{p}_n$ but the estimator $\hat{p}_n$ itself. This makes it adaptive to the contamination proportion $\epsilon$ and we will show that this costs us only a constant factor in the asymptotic rate. Secondly, we follow \citet{donoho1996density} and pick the following values for the resolution levels $j_0$, $j_1$, 
        \begin{align*}
            2^{j_0} &= n^{\frac{1}{D+2r}}\\
            2^{j_1} &= \left(\frac{n}{\log n}\right)^{1/D}
        \end{align*}
    where $r$ is the regularity of the wavelets used to construct the MRA defined above. We can decompose the error as 
		\begin{align*}
	    \E d_{\F}&\left(\hat{p}_n, p\right)\\
            &\leq \E d_{\F}
                \left(
                \sum_{ k\in \Z} \hat{\alpha}_{k} \phi_{k} +
                \sum_{j=0}^{j_0} \sum_{\psi\in \Psi_j} \hat{\beta}_{\psi}\psi,
		        \sum_{ k\in \Z} \alpha^p_{k}\phi^p_{k} +
		        \sum_{j=0}^{j_0} \sum_{\psi\in \Psi_j}\beta^p_{\psi}\psi
                \right) \\
            &+ d_{\F}\left(
            \sum_{j= j_0}^{j_1}\sum_{\psi\in \Psi_j} \tilde{\beta}_{\psi} \psi
            , \sum_{j= j_0}^{j_1}\sum_{\psi\in \Psi_j} \beta^p_{\psi} \psi 
            \right)\\
            &+ d_{\F}\left(
            \sum_{ k\in \Z} \alpha^p_{k}\phi_{k}+\sum_{j=0}^{j_1} \sum_{\psi\in \Psi_j} 
		           \beta^p_{\psi}\psi
            , p
            \right)\\
            &+ \epsilon d_{\F}\left(\sum_{ k\in \Z} 
                \alpha^g_{k}\phi_{k}
                +\sum_{j=0}^{j_1} \sum_{\psi\in \Psi_j} 
		           \beta^g_{\psi}\psi
		      ,0
		    \right)\\
		    &+ \epsilon d_{\F}\left(\sum_{ k\in \Z} 
                \alpha^p_{k}\phi_{k}
                +\sum_{j=0}^{j_1} \sum_{\psi\in \Psi_j} 
		           \beta^p_{\psi}\psi
		      ,0
		    \right)
    \end{align*}
    where we have an extra term at end as opposed to the non-adaptive case above. Since the density $p$ is bounded above this term is bounded by $\epsilon$ and hence by the misspecification error. The bound on the misspecification error does not change since it does not depend on the values of $j_0$ or $j_1$.
    
    Now, since, $\sigma_g<r$ we know that the number of linear terms or $j_0$ has is smaller than above. Moreover, since $\sigma_g> D/p_g$ the number of non-linear terms $j_1$ is larger than above. Therefore, it is clear that the bias and the variance bounds hold as above. It only remains to bound the non-linear terms which from lemma~\ref{lemma:non_linear_terms} amounts to bounding
    \begin{align}
          \sum_{j=j_0}^{j_1}  
        2^{-j(\sigma_d+D/2-D/p_d)} 
        \frac{\norm{\beta_j^p}_s^{s/p_d'}+
        \epsilon^{s/p_d'}\norm{ \beta_j^g}_{s}^{s/p_d'}}
        {\sqrt{n}^{1-s/p_d'}}.
    \end{align}
    which following the same procedure as above is bounded above by 
    \begin{align*}
            \sum_{j=j_0}^{j_1}  
	        &2^{-j(\sigma_d+D/2-D/p_d)}
         \frac{\norm{\beta_j^p}_{s}^{s/p_d'}}{\sqrt{n}^{1-s/p_d'}}
         + \epsilon\sum_{j=j_0}^{j_1}  
	        2^{-j(\sigma_d+D/2-D/p_d)}2^{Dj(1/p_d'-1/2)}\\
	 \end{align*}
    where the second term is bounded by the misspecification error. Now the first term is the same as in the case of uncontaminated setting and thereby we can bound it in the same way. 
    
    When $(2\sigma_g+D)p_g\leq (D-2\sigma_d)p_d'$ let $s=-p_d'\frac{2\sigma_d+D-2D/p_d}{2\sigma_g+D-2D/p_g}$. Note that 
    \[
        p_g\leq -p_d'\frac{2\sigma_d+D-2D/p_d}{2\sigma_g+D-2D/p_g} \leq p_d'
    \]
    where the first inequality is equivalent to the condition above and the second is equivalent to $\sigma_g\geq -\sigma_d$.
    We have the following bound, when we pick ,
    \begin{align*}
        \sum_{j=j_0}^{j_1}  
	        &2^{-j(\sigma_d+D/2-D/p_d)}
         \frac{2^{-j(\sigma_g+D/2-D/p_g)s/p_d'}}{\sqrt{n}^{1+\frac{2\sigma_d+D-2D/p_d}{2\sigma_g+D-2D/p_g}}}\\
         &\asymp n^{-\frac{\sigma_g+\sigma_d+D/p_d'-D/p_g}{2\sigma_g+D-2D/p_g}}
    \end{align*}
    as desired (where we omit any $\log(n)$ terms).
    
    Now, when $(2\sigma_g+D)p_g\geq (D-2\sigma_d)p_d'$ the error of the non-linear terms is bounded by the error of the first non-linear term i.e. $j=j_0$. We can bound the error of the non-linear terms for all $j\geq j_0^*$ where $j_0^*$ is the original non-adaptive threshold i.e. $2^{j_0}= n^{\frac{1}{2\sigma_g+D}}$. For the extra terms between $j_0$ and $j_0^*$ we show that in this range the error of the non-linear terms cannot be worse that the linear error. The large deviation terms (1) and (2) above are negligible by the same argument as above. The terms (3) and (4) can be trivially bounded by 
    \[
        2^{-j(\sigma_d+D/2-D/p_d)}\frac{2^{Dj/p_d}}{\sqrt{n}} = \frac{1}{\sqrt{n}}2^{j(D/2-\sigma_d)}
    \]
    which is bounded by the linear rate.

\section{Lower Bounds}
     In this section we prove our lower bounds. We first provide lower bounds in the case of structured contamination since these also hold for the case of unstructured contamination. We then provide additional lower bounds that are specific to the unstructured case.
     
     \subsection{Structured Contamination}
     We assume here that $G$ has a density that lives in a Besov space i.e. $\F_c = B^{\sigma_c}_{p_c,q_c}$.
     
     \begin{proof}
     We will use Fano's lemma to imply lower bounds here. 
     
     First we show that the lower bounds on the risk in the setting of no contamination also bound the risk in the contaminated setting.
     The key idea here is that if the set of densities chosen to provide bounds in the uncontaminated setting (when $\epsilon=0$) are perturbations of a ``nice'' density $p_0$, then in the contaminated setting we can choose our contamination density to be this nice density $g_0$. This will imply that the contamination does not affect the samples i.e. the samples are generated merely from the perturbation (since $(1-\epsilon)(g_0+p_\tau)+\epsilon g_0=g_0+(1-\epsilon)p_\tau$ where $p_\tau$ is some perturbation). 
     
     We first state Fano's lemma.
	
	\begin{lemma} ({\bf Fano's Lemma}; Simplified Form of Theorem 2.5 of \cite{tsybakov2009introduction})
	Fix a family $\P$ of distributions over a sample space $\X$ and fix a pseudo-metric $\rho : \P \times \P \to [0,\infty]$ over $\P$. Suppose there exists a set $T \subseteq \P$ such that there is a $p_0\in T$ with $p\ll p_0$ $\forall p\in T$ and  
        \[s := \inf_{p,p' \in T} \rho(p,p') > 0
          \quad \text{ , } \quad
          \sup_{p \in T} D_{KL}(p,p_0)
          \leq \frac{\log |T|}{16},\]
    where $D_{KL} : \P \times \P \to [0,\infty]$ denotes Kullback-Leibler divergence.
    Then,
        \[\inf_{\hat p} \sup_{p \in \P} \E \left[ \rho(p,\hat p) \right]
          \geq \frac{s}{16}\]
    where the $\inf$ is taken over all estimators $\hat p$. 

    \end{lemma}
    
    Now we choose our set of densities as,
     \begin{align*}
        p &= p_0 &p_\tau &= p_0 + \frac{1}{(1-\epsilon)}c_g f_{\tau} \\
        g &= p_0 &g_\tau &= p_0.
    \end{align*}
    where $p_0+c_g f_\tau \in \F_g$ for every $\tau$. Notice that the KL divergence remains unchanged from the uncontaminated setting, 
    \begin{align*}
        KL((1-\epsilon)p_0+\epsilon p_0 &, (1-\epsilon)( p_0 + \frac{1}{1-\epsilon}c_g g_{\tau})+\epsilon p_0) \\
         &= KL (p_0, p_0 + c_g f_\tau)
    \end{align*}

    i.e. the KL divergence doesn't depend on the existence of contamination. Neither does $d_{\F_d}(p_\tau, p_{\tau'})$. Since, $1-\epsilon\in [1/2,1]$ we can treat it as a constant and only write $c_g$ henceforth.
    Therefore we are essentially in the case of no contamination i.e. if there exist densities $p, p_{\tau}$  indexed by $\tau$ such that they satisfy the assumptions of Fano's lemma then the conditions of Fano's lemma are also satisfied for $(1-\epsilon)p+\epsilon g$, $(1-\epsilon)p_\tau + \epsilon g_\tau$. Moreover, the distance we want to bound i.e. $d_{\F_d}(p_\tau, p_{\tau'})$ does not depend on the contamination either. Therefore, we have a lower bound here that is the same as the one in the setting with no contamination. 
    
    We note that the densities used to prove the lower bound in \cite{uppal2019nonparametric} are exactly of this form (see section B of the appendix) (\cite{uppal2019nonparametric} study the uncontaminated version of this problem). Therefore, their lower bound (see Theorem 4) is implied here i.e. 
    \[
        C\left(
		\frac{1}{\sqrt{n}}+
		n^{-\frac{\sigma_g+\sigma_d}{2\sigma_g+D}}+ \left(\frac{\log n}{n}\right)^{\frac{\sigma_g+\sigma_d-D/p_g+D/p_d'}{2\sigma_g+D-2D/p_g}}
		\right)
    \]
    
    Second, we consider the case where we ``move'' the perturbation so that the samples are generated from the same density. In particular, we first perturb the contamination and then we move this perturbation to the true density i.e. 
        \begin{align*}
        p &= p_0 &\tilde{p} &= p_0 + \frac{\epsilon}{1-\epsilon}c\psi_{\epsilon} \\
        g &= g_0 + c\psi_{\epsilon}
        &\tilde{g} &= g_0
        \end{align*}
    Then the KL divergence between the densities that generate the samples is zero since they are the same ($(1-\epsilon)p+\epsilon g$ is the same in both cases). It is easy to see that $\tilde{p}, g$ both live in the respective density classes i.e. $\F_g, \F_c$ for a small enough constant $c$. Using Le Cam's two point argument i.e. 
    
    \begin{lemma}(Le Cam (see section 2.3 of \cite{tsybakov2009introduction}))
        Let $P_1$, $P_2$ be two probability measures on $\X$ s.t. $d(P_1, P_2) = s$. If $KL(P_1, P_2) \leq \alpha<\infty$ then, for any $\hat{P}$
        \[
            \E_{P_i} [d(\hat{P}, P_i)]\geq \frac{s}{8} e^{-\alpha}
        \]
    \label{lemma:lecam}
    \end{lemma}
    
    we have a lower bound that is the distance between $p_0$ and $p_0+\frac{\epsilon}{1-\epsilon}\psi_\epsilon $ i.e. 
    \[
        d_{\F_d}(p_0, p_0+\frac{\epsilon}{1-\epsilon}\psi_\epsilon) = \epsilon.
    \]  
    \end{proof}
    This section provided lower bounds on the risk that are minimax in the structured contamination setting. We now provide additional bounds that hold when we have no structural assumptions on the contamination.
    
    \subsection{Unstructured Contamination}
        Here we assume only that $g$ is a compactly supported probability density. We will pick a single perturbation of a ``nice'' density and use this to construct the contamination densities in such a way that the data is generated from the same density. Hence, the KL divergence between the data generating densities will be zero. Then, as before, we can apply Le Cam's two point argument to bound the risk. 

    \subsubsection{Sparse or Lower Smoothness Case}
        Let $p = g_0$, $\tilde{p} = g_0+c_g\psi_0$ for some $\psi_0\in \Psi_j$. Now we can pick densities $g, \tilde{g}$ such that 
        \[
            (1-\epsilon)p+\epsilon g = (1-\epsilon)\tilde{p}+\epsilon\tilde{g}
        \]
        if and only if 
        \[
            g-\tilde{g} = \frac{(1-\epsilon)}{\epsilon}(\tilde{p}-p)
        \]
        integrates to zero and its $L_1$ norm is $\leq 2$ (see Lemma 6.6 of Liu and Gao). For $\tilde{p}$ to be a density in $\F_g$ we need 
        \[
            c_g \leq c\min (2^{-Dj/2},2^{-j(\sigma_g+D/2-D/p_g)})
        \]
        Since $\sigma_g\geq D/p_g$, we let $c_g = 2^{-j(\sigma_g+D/2-D/p_g)}$.
        From the above constraint on the $L_1$ of $g-\tilde{g}$ norm we need 
        \[
            \frac{c_g}{\epsilon}2^{-Dj/2}\norm{\psi}_{\infty} \leq 2
        \]
        This is equivalent to
        \[
            2^{-j(\sigma_g+D-D/p_g)}\leq c \epsilon
        \]
    where $c$ is a constant. We pick $2^{j} = \epsilon^{-\frac{1}{\sigma_g+D-D/p_g}}$. 
    We also choose a simple discriminator i.e. 
        \[
            \Omega_d = \{c_d \psi_0\}
        \]
    where $c_d = 2^{-j(\sigma_d+D/2-D/p_d)}$ so that $\Omega_d\subseteq \F_d$. Then, by~\ref{lemma:lecam} the minimax risk is lower bounded by
        \begin{align*}
                d_{\F_d}(p, \tilde{p})
                    &\geq d_{\Omega_d}(p, \tilde{p})\\
                &\gtrsim c_g c_d\\
                &= c2^{-j(\sigma_g+\sigma_d+D-D/p_g-D/p_d)} \\
                &= \epsilon^{\frac{\sigma_g+\sigma_d+D/p_d'-D/p_g}{\sigma_g+D-D/p_g}}.
        \end{align*}

\end{document}